\documentclass[11pt, a4paper,leqno]{amsart}

\usepackage{amsmath,amsthm,amscd,amssymb,amsfonts, amsbsy}
\usepackage{latexsym}
\usepackage{txfonts}
\usepackage{exscale}
\usepackage{mathrsfs}
\usepackage{amsmath}
\usepackage{float}

\usepackage{esint} % For \fint symbol
\usepackage{enumerate}

\usepackage[colorlinks=true, linkcolor=blue, citecolor=red, urlcolor=red]{hyperref} % Add coloured links
\usepackage{fancyhdr}

\usepackage[utf8]{inputenc}

\calclayout
\allowdisplaybreaks

\numberwithin{equation}{section}

\theoremstyle{plain}
\newtheorem{theorem}[equation]{Theorem}
\newtheorem{thm}[equation]{Theorem}

\newtheorem{lem}[equation]{Lemma}

\newtheorem{prop}[equation]{Proposition}

\theoremstyle{definition}

\newtheorem{defn}[equation]{Definition}

\theoremstyle{remark}

\newtheorem{rem}[equation]{Remark}

\makeatletter
\@namedef{subjclassname@2020}{%
  \textup{2020} Mathematics Subject Classification}
\makeatother

\def\N{\mathbb {N}}
\def\R{\mathbb{R}}

\newcommand{\rn}{\mathbb{R}^n}

\renewcommand{\emptyset}{\mbox{\textup{\O}}}

\newcommand{\normmm}[1]{{\left\vert\kern-0.25ex\left\vert\kern-0.25ex\left\vert #1
		\right\vert\kern-0.25ex\right\vert\kern-0.25ex\right\vert}}

\def\Z{\mathbb{Z}}
\def\S{\mathbb{S}}

\begin{document}

\title[Affine  Moser--Trudinger inequalities]{New approach to
affine Moser--Trudinger inequalities \\via Besov polar projection bodies}

\author[Oscar Dom\'{i}nguez et al.]{Oscar Dom\'{i}nguez, Yinqin Li, Sergey Tikhonov,
Dachun Yang and Wen Yuan}

\address{(O. Dom\'{i}nguez) Departamento de M\'etodos Cuantitativos\\CUNEF Universidad\\28040 Madrid\\ Spain}
\email{oscar.dominguez@cunef.edu}

\address{(Y. Li, D. Yang, W. Yuan) Laboratory of Mathematics and Complex
Systems (Ministry of Education of China),
School of Mathematical Sciences, Beijing Normal
University, Beijing 100875, The People's Republic of China}
\email{yinqli@mail.bnu.edu.cn, dcyang@bnu.edu.cn, wenyuan@bnu.edu.cn}

\address{(S. Tikhonov) Centre de Recerca Matem\`{a}tica, Campus de Bellaterra, Edifici C 08193 Bellaterra (Barcelona),
Spain; ICREA, Pg. Llu\'{i}s Companys 23, 08010 Barcelona, Spain, and Universitat Aut\`{o}noma de Barcelona,
Facultat de Ci\`{e}ncies, 08193 Bellaterra, Barcelona, Spain}
\email{stikhonov@crm.cat}
%
%\address{D. Yang, Laboratory of Mathematics and Complex
%Systems (Ministry of Education of China),
%School of Mathematical Sciences, Beijing Normal
%University, Beijing 100875, The People's Republic of China}
%\email{dcyang@bnu.edu.cn}
%
%\address{W. Yuan, Laboratory of Mathematics and Complex
%Systems (Ministry of Education of China),
%School of Mathematical Sciences, Beijing Normal
%University, Beijing 100875, The People's Republic of China}
%\email{}

\thanks{\emph{Acknowledgements.} Oscar Dom\'{i}nguez is supported by the AEI grant RYC2022-037402-I, Yinqin Li, Dachun Yang, and
Wen Yuan are supported by the National Key Research and Development Program of China
(Grant No.\ 2020YFA0712900) and the National
Natural Science Foundation of China
(Grant Nos.\ 12371093, 12071197 and 12122102), and Sergey Tikhonov is supported
by PID2020-114948GB-I00, 2021 SGR 00087, AP 14870758,
the CERCA Programme of the Generalitat de Catalunya, and Severo Ochoa and Mar\'{i}a de Maeztu
Program for Centers and Units of Excellence in R\&D (CEX2020-001084-M)}


%%%%%%%%%%%%%%%%%%%%%%%%%%%%%%%%%%%%%%%%%%
%%%%%%%%%%%%%%%%%%%%%%%%%%%%%%%%%%%%%%%%%%%

\begin{abstract}

We extend the affine inequalities on $\R^n$ for Sobolev functions in $W^{s, p}$ with $1 \leq p < \frac{n}{s}$ obtained recently  by Haddad--Ludwig \cite{HaddadLudwig, HaddadLudwig2} to the remaining range $p \geq \frac{n}{s}$. For each value of $s$, our results  are stronger than  affine Moser--Trudinger and Morrey inequalities. As a byproduct, we establish the analog of the classical $L^p$ Bourgain--Brezis--Mironescu inequalities  related to the Moser--Trudinger case $p=n$.  Our main tool is the affine invariant provided by Besov  polar projection bodies.
	
\end{abstract}

\maketitle

\setcounter{tocdepth}{1}
%\tableofcontents

%%%%%%%%%%%%%%%%%%%%%%%%%%%%%%%%%%%%%%%%%%%%

\section{Introduction}

\subsection{Affine and fractional Sobolev inequalities: An overview}

The classical Sobolev inequality on $\R^n$ states that, for
 every $f \in W^{1}_p$ ($1 \leq p < n$)\footnote{Otherwise is stated, all function spaces considered in this paper are based on $\R^n$.},
 %the Sobolev space\footnote{Otherwise is stated, all function spaces considered in this paper are based on $\R^n$.} of $L^p$ functions such that the Euclidean norm of  their gradient belongs to $L^p$,

\begin{equation}\label{Intro1}\tag{S}
	\|f\|_{L^{\frac{n p}{n-p}}}^p \leq C_{n, p} \, \int_{\R^n} |\nabla f (x)|^p \, dx.
\end{equation}
 This inequality has a long and rich history. We only mention that the  optimal values of the  constant $C_{n, p}$ and the  corresponding extremizers  are known, see  \cite{Aubin, Talenti} and also \cite{Rodemich}.
 
 Let us  focus on two different improvements of inequality \eqref{Intro1} that have
 evolved into independent research fields:
 the affine refinement of \eqref{Intro1} by Lutwak--Yang--Zhang
and the  Bourgain--Brezis--Mironescu approach via fractional Sobolev spaces.

Lutwak, Yang, and Zhang (cf. \cite{Zhang}  and \cite{LutwakYangZhang} for $p=1$ and $1 < p < n$, respectively) showed that \eqref{Intro1} can be significantly improved as follows:\footnote{Let $B^n$ be the $n$-dimensional Euclidean unit ball and $\omega_n = |B^n|$.}
\begin{equation}\label{Intro2}\tag{LYZ}
	\|f\|_{L^{\frac{n p}{n-p}}}^p \leq C_{n, p} \frac{n \omega_n^{\frac{n + p}{n}}}{\alpha_{n, p}} \, |\Pi^*_p f|^{-\frac{p}{n}} \leq C_{n, p}  \int_{\R^n} |\nabla f (x)|^p \, dx,
\end{equation}
where $|\Pi^*_p f|$ is the $n$-dimensional Lebesgue measure of $\Pi^*_p f$,  the $L^p$ polar projection body\footnote{A certain star body in $\R^n$ associated to $f$, cf. Section \ref{SectionPrelimPolar} for the precise definition.} of $f$ and\footnote{$\sigma$ is the $(n-1)$-dimensional Hausdorff measure on $\S^{n-1}$.} \begin{equation}\label{DefAlpha}\alpha_{n, p} = \int_{\S^{n-1}} |e \cdot \xi|^p \, d \sigma(\xi), \qquad e \in \S^{n-1}.
\end{equation} The first inequality in \eqref{Intro2} is nowadays referred to as the \emph{affine} Sobolev inequality since involved quantities (i.e., Lebesgue norms and volumes of $L^p$ polar projection bodies) that remain invariant under volume-preserving affine transformations. This is in sharp contrast with the classical inequality \eqref{Intro1}, where standard Sobolev norms are only  invariant under rigid motions.

The inequalities \eqref{Intro2} (more precisely, the Zhang's affine Sobolev inequality with $p=1$) are of special relevance in Brunn--Minkowski theory. Indeed, if $E \subset \R^n$ is a convex set, then $\Pi^*_1 \mathbf{1}_E$ coincides with the classical notion of the polar projection body  $\Pi^* E$. Then the famous Petty projection inequality follows from \eqref{Intro2} (more precisely, its extension to functions of total variation obtained in \cite{Wang}) with $f= \mathbf{1}_E$,
\begin{equation}\label{Petty}
	\bigg(\frac{|E|}{|B^n|} \bigg)^{\frac{n-1}{n}} \leq \bigg(\frac{|\Pi^* E|}{|\Pi^* B^n|} \bigg)^{-\frac{1}{n}} \leq \frac{P(E)}{P(B^n)},
\end{equation}
where $P(E)$ stands for the perimeter of $E$. In particular, the Petty projection inequality is stronger than the classical isoperimetric inequality. Hence \eqref{Intro2} may be considered as a functional analog of \eqref{Petty}.

Taking a different approach,
 %In a completely different direction than \eqref{Intro2},
 Bourgain, Brezis, and Mironescu  showed that the Sobolev inequality \eqref{Intro1} can also be sharpened  through suitable normalizations of fractional Sobolev norms.
  Let us first recall  that %briefly revisit  this idea: the well-known issue with
     the family of Gagliardo (semi-)norms
 \begin{equation}\label{ClassicGag}
	\|f\|_{W^{s, p}} = \bigg( \int_{\R^n} \int_{\R^n} \frac{|f(x)-f(y)|^p}{|x-y|^{s p + n}} \, dx \, dy \bigg)^{\frac{1}{p}}, \qquad s \in (0, 1),
\end{equation} % is their lack of continuity at the endpoint $s=1$, that is, they
does not converge at the expected  norm as  $s\to 1$, namely,
$$
\lim_{s \to 1-} \|f\|_{W^{s, p}} \neq \bigg(\int_{\R^n} |\nabla f (x)|^p \, dx \bigg)^{\frac{1}{p}}.
$$
In fact, the failure is even more dramatic since letting $s=1$ (at least formally)  in \eqref{ClassicGag} gives (cf. \cite{B02})
$$
	\|f\|_{W^{1, p}} < \infty \implies f \quad \text{is a constant}.
$$
Bourgain, Brezis, and Mironescu \cite{BBM} overcame this obstruction through  a  renormalization of \eqref{ClassicGag}. More precisely, if $f \in W^1_p$ then
\begin{equation}\label{Intro3}
	\lim_{s \to 1-} p (1-s) \|f\|_{W^{s, p}}^p = \alpha_{n, p} \int_{\R^n} |\nabla f (x)|^p \, dx,
\end{equation}
where $\alpha_{n, p}$ is given by \eqref{DefAlpha}. Furthermore, the prefactor $1-s$ in \eqref{Intro3}  corresponds to the optimal behaviour as $s \to 1-$ of the constant related to the classical Sobolev embedding theorem $W^{s, p} \subset L^{\frac{n p}{n-s p}}, \, p \in [1, n/s)$. Specifically, according\footnote{To be more precise, the authors obtained \eqref{Intro4} for the equivalent case of mean zero functions over a fixed cube.}  to \cite{BBM02},
\begin{equation}\label{Intro4}\tag{BBM}
	\|f\|_{L^{\frac{n p}{n-s p}}}^p \leq \sigma_{s, n, p} \, \|f\|_{W^{s, p}}^p \qquad \text{with} \qquad \sigma_{s, n, p} = c \,  \frac{1-s}{(n-s p)^{p-1}},
\end{equation}
where $c$ is a constant\footnote{In general, the optimal value of $c$ remains unknown.} independent of $s \in [1/2, 1)$.
An extension of  \eqref{Intro4}, taking now into account the behaviour of the constant as $s \to 0+$, was later obtained in \cite{Mazya}; cf. also \cite{FrankSeiringer, KMX, KL}. Putting together \eqref{Intro3} and \eqref{Intro4}, one arrives at the classical inequality \eqref{Intro1}  (possibly with a non optimal constant). In this sense, \eqref{Intro4} can be viewed as a refinement of \eqref{Intro1} using fractional smoothness.

The connection between  the two approaches  described above % (i.e., \eqref{Intro2} and \eqref{Intro4})
  has been recently discovered by Haddad and Ludwig, see \cite{HaddadLudwig2} and \cite{HaddadLudwig} for the case $p=1$ and $1 < p <n$, respectively.
  There the authors introduced the concept of the fractional polar projection body $\Pi^{*, s}_p f$
  %To proceed with, the authors start by introducing the concept of fractional polar projection body
 related to $f \in W^{s, p}$ (cf. Section \ref{SectionPrelimPolar}) as the fractional counterpart of the classical $\Pi^*_p f$ introduced by Lutwak--Yang--Zhang.
  The following affine version  of  formula \eqref{Intro3} was shown to be true (see \cite[Theorem 3]{HaddadLudwig2} and \cite[Theorem 10]{HaddadLudwig}): For every $f \in W^1_p, p \in [1, \infty)$,
\begin{equation}\label{HLLimit}
	\lim_{s \to 1-} p (1-s) |\Pi^{*, s}_p f|^{-\frac{p s}{n}} = |\Pi^*_p f|^{-\frac{p}{n}}.
\end{equation}
Moreover, for $0 < s < 1$  and $1 \leq p < n/s$,
\begin{equation}\label{Intro5}\tag{HL}
	\|f\|_{L^{\frac{n p}{n-s p}}}^p \leq \sigma_{s, n, p} \, n \omega_n^{\frac{n + s p}{n}} |\Pi^{*, s}_p f|^{-\frac{p s}{n}},
\end{equation}
where $\sigma_{s, n, p}$ is the constant appearing in  \eqref{Intro4}, cf. Theorem 1 in \cite{HaddadLudwig2} and \cite{HaddadLudwig}. The inequality \eqref{Intro5} is affine (in the sense that both sides are invariant under volume-preserving linear transformations) and improves simultaneously both  \eqref{Intro2} and \eqref{Intro4}. Indeed, in light of \eqref{HLLimit},  \eqref{Intro2} (up to optimal constants) follows automatically from \eqref{Intro5} by taking limits as $s \to 1-$. On the other hand, a simple application of the dual mixed volume inequality (or more generally, H\"older's inequality) gives
\begin{equation}\label{RelIntro}
	n \omega_n^{\frac{n + s p}{n}} |\Pi^{*, s}_p f|^{-\frac{p s}{n}} \leq \|f\|_{W^{s, p}}^p,
\end{equation}
so \eqref{Intro5} is stronger than \eqref{Intro4}.

For the reader's convenience,  in Table \ref{Table1} we provide a quick guide to the interrelation between the inequalities \eqref{Intro1}, \eqref{Intro2}, \eqref{Intro4}, and \eqref{Intro5}.

\begin{table}[H]
  \centering
\begin{tabular}{c|c|c}
  \eqref{Intro1} & $1\leftarrow s$&  \eqref{Intro4} \\
  $\displaystyle{C_{n, p}}\int_{\mathbb{R}^n}|\nabla f(x)|^p
  \,dx\ge\|f\|^p_{L^\frac{n}{n-p}} $ & $\Leftarrow$ &
  $\displaystyle{\sigma_{s, n, p} \|f\|^p_{W^{s, p}} \ge\|f\|_{L^\frac{n p}{n-s p}}}$ \\\hline
  $\Uparrow$ & \ & $\Uparrow$ \\\hline
  \eqref{Intro2} & $1\leftarrow s$ & \eqref{Intro5}\\
  $\displaystyle{C_{n, p} \frac{n \omega_n^{\frac{n + p}{n}}}{\alpha_{n, p}} \left|\Pi^*_p f\right|^{-\frac{p}{n}}
  \ge \|f\|^p_{L^{\frac{n p}{n-p}}}}$ & $\Leftarrow$ &
  $\displaystyle{\sigma_{s, n, p} n \omega_n^{\frac{n+s p}{n}} \left|\Pi^{*,s}_pf\right|^{-\frac{ p s}{n}}
  \ge \|f\|^p_{L^{\frac{n p}{n-s p}}}}$
\end{tabular}
\vspace{2mm}

  \caption{Affine and fractional Sobolev inequalities.}
  \label{Table1}
\end{table}

The limiting case of \eqref{Intro1} when\footnote{Note that $W^1_1 \subset L^\infty$ if  $n=1$. So, throughout this paper, we assume $n \geq 2$. } $p=n > 1$ corresponds to the celebrated Moser--Trudinger inequality \cite{Moser, Trudinger}. Namely, there exist constants $C$ and $A$ such that\footnote{As usual, $p'$ denotes the dual exponent of $p \in [1, \infty]$ given by $p' = \frac{p}{p-1}$.}
\begin{equation}\label{Intro6}\tag{MT}
	\frac{1}{|\text{supp } f|} \int_{\R^n} \left[ \exp \bigg(C \, \frac{|f(x)|}{\|\nabla f\|_{L^n}} \bigg)^{n'} - 1\right] \, dx \leq A
\end{equation}
for every $f \in W^1_n$ with $0 < |\text{supp } f| < \infty$.  In other words, functions in $W^1_n$ might be unbounded but they are exponentially integrable of order $n'$. The value  $C= n   \omega_n^{1/n}$ is the best possible one for which \eqref{Intro6} holds for some $A$. The affine version of \eqref{Intro6} is due to Cianchi, Lutwak, Yang, and Zhang \cite[Theorem 1.1]{CLYZ}, who showed the existence of an explicit constant $c_n$, depending only on $n$,  such that
\begin{equation}\label{Intro7}\tag{CLYZ}
\frac{1}{|\text{supp } f|} \int_{\R^n} \left[ \exp \bigg(C \, \frac{|f(x)|}{\|\nabla f\|_{L^n}} \bigg)^{n'} - 1\right] \, dx	\leq \frac{1}{|\text{supp } f|} \int_{\R^n} \left[ \exp \bigg(C \, \frac{|f(x)|}{c_n |\Pi^*_n f|^{-1/n}} \bigg)^{n'} - 1\right] \, dx  \leq A.
\end{equation}

The super-limiting case of \eqref{Intro1} corresponds  to classical Morrey inequalities.  Recall that the Morrey inequality provides the quantitative version of the embedding $W^1_p \subset L^\infty$ with $p > n$:
\begin{equation}\label{Intro8}\tag{M}
	\|f\|_{L^\infty} \leq c |\text{supp } f|^{\frac{1}{n}-\frac{1}{p}} \|\nabla f\|_{L^p},
\end{equation}
for every $f \in W^1_p$ with $|\text{supp }f| < \infty$ (see e.g. \cite{Talenti}). Here, $c = n^{-1/p} \omega_n^{-1/n} \big(\frac{p-1}{p-n} \big)^{1/p'}$. The affine counterpart of \eqref{Intro8}  was also obtained by Cianchi--Lutwak--Yang--Zhang \cite[Theorem 1.2]{CLYZ}: There exists a constant $c_{n, p}$, depending only on $n$ and $p$, such that
\begin{equation}\label{Intro9}\tag{CLYZM}
	\|f\|_{L^\infty} \leq c |\text{supp }f|^{\frac{1}{n}-\frac{1}{p}} c_{n, p} |\Pi^*_p f|^{-1/n} \leq  c |\text{supp }f|^{\frac{1}{n}-\frac{1}{p}} \|\nabla f\|_{L^p}.
\end{equation}

\subsection{Main results}
The goal of this paper is to complement the affine fractional Sobolev inequalities with $1 \leq p < n/s$ given in \cite{HaddadLudwig, HaddadLudwig2} for both the limiting case $p=n/s$ and the super-limiting case $p > n/s$.
Before we state our main results, let us mention
several essential difficulties
that arise when trying to apply the methodology developed in \cite{HaddadLudwig, HaddadLudwig2} to the missing case $p\geq n/s$.

\subsubsection{\textbf{The case $p \geq n/s$.}}\label{SectionObs}
  In short, Haddad and Ludwig
  ingeniously reduced the proof of \eqref{Intro5} to the classical setting given by \eqref{Intro4} with $p < n/s$ by applying the symmetrization technique and  P\'olya--Szeg\H{o}-type inequalities for
  fractional polar projection bodies.
 However, the following important  question remains open:
how to obtain  BBM type phenomenon\footnote{Connections between Moser--Trudinger inequalities and certain (non-local, non-convex) variants of Gagliardo norms have been recently obtained by Mallick and Nguyen \cite{Nguyen}.}
 for % analogs of  \eqref{Intro4} related to
 \eqref{Intro6} (i.e., the analog of the \eqref{Intro4} inequality with $p=n/s$ and $s \to 1-$)?
 %  However, an outstanding open question is that  full\footnote{Connections between Moser--Trudinger inequalites and certain (non-local, non-convex) variants of Gagliardo norms have been recently obtained by Mallick and Nguyen \cite{Nguyen}.} analogs of  \eqref{Intro4} relative to \eqref{Intro6} and \eqref{Intro8} are missing in the literature!
 A crucial observation is that,  in some sense, the scale of fractional Sobolev spaces $W^{s, p}$ is not rich enough to capture  the underlying essence of  \eqref{Intro6}.

  One of our main ideas is to show that situation seriously improves when
   the Sobolev spaces $W^{s, p}$ are replaced by
the more general scale of Besov spaces $B^s_{p, q}$, cf. Section \ref{SectionBesov}. Here we recall  that $B^s_{p, p} = W^{s, p}$.

It turns out that  the integrability parameter $q$ in Besov spaces allows us to %plays a fundamental role since it provides us with the additional strength required to
 extend the Bourgain--Brezis--Mironescu method to sharpen Moser--Trudinger inequalities.
 At a more technical level, it can be easily observed  that functions in $W^{n/p, p}, p > n,$ do not necessarily satisfy the exponential integrability of the same order as in \eqref{Intro6} (i.e.,  $n'$), but only the weaker exponential integrability of order $p'$.
  However, this issue may be overcome if we consider $B^{n/p}_{p, n}$  in place of $W^{n/p, p}$, cf. \eqref{PMT} below.\footnote{Note that $B^{n/p}_{p, n} \subset B^{n/p}_{p, p}= W^{n/p, p}$ since $n < p$.}

\subsubsection{
\textbf{A new approach via Besov polar projection bodies.}} In Section \ref{Section3}, we use Besov polar projection bodies as a tool that enables us to overcome the obstructions explained in Subsection \ref{SectionObs}. More precisely, we introduce Besov polar projection bodies $\Pi^{*, s}_{p, q} f$ related to the Besov norm  $\|f\|_{B^s_{p, q}}$ %are introduced
 as a natural extension of $\Pi^{*, s}_p f$ for $\|f\|_{W^{s, p}}$; cf. Definition \ref{DefBPPB}.  Then we show that $\Pi^{*, s}_{p, q} f$ is an affine invariant concept that is intimately connected with anisotropic Besov spaces $B^s_{p, q; K}$ defined by (cf. Definition \ref{DefABs})
$$
		\|f\|_{B^s_{p, q; K}} = \bigg(\int_0^\infty t^{-s q} \omega_K(f, t)_p^q \, \frac{dt}{t} \bigg)^{\frac{1}{q}},
$$
where $K$ is a star body in $\R^n$ with corresponding  $L^p$-moduli of smoothness (cf. Definition \ref{DefModSec})
$$
	\omega_K (f, t)_p =\bigg(\frac{1}{t^n |K|} \int_{\|h\|_K < t} \|f(\cdot + h) -f\|^p_{L^p} \, d h \bigg)^{\frac{1}{p}}.
$$
In the special case when $K = B^n$ we recover the classical Besov space $B^s_{p, q}$ and the modulus of smoothness $\omega(f, t)_p$, respectively. We will see  that $\omega_K(f, t)_p$ is a natural object to be investigated from the point of view of convex geometry. In particular, the anisotropic Sobolev spaces \cite{Ludwigb, Ludwig} based on the polar moment body of $K$
 can be characterized in terms of  $\omega_K(f, t)_p$  (cf. Proposition \ref{Prop3.13}) and  the  P\'olya--Szeg\H{o} inequality holds for  $\omega_K(f, t)_p$  (cf. Proposition \ref{ThmPS}.)

\subsubsection{\textbf{Affine fractional Moser--Trudinger inequalities.}} In Section \ref{Section5}, we introduce the following affine invariant
\begin{equation}\label{NA}
	\mathcal{G}_r (f) := \bigg(1-\frac{n}{r} \bigg)^{\frac{1}{n}}  \big|\Pi^{*, \frac{n}{r}}_{r, n} f \big|^{-\frac{1}{r}}, \qquad r > n.
\end{equation}
The normalization constant $(1-\frac{n}{r})^{\frac{1}{n}}$ in \eqref{NA} will play a key role in our results.
Then we establish the following affine fractional Moser--Trudinger inequalities.

\begin{thm}\label{ThmFractCLYYIntro}
	Let $f \in B^{n/r}_{r, n}, r > n$, with $0 < |\emph{\text{supp }}f| < \infty$. Then there exists a constant $c_n$, depending only on $n$, such that for every $\beta > c_n$ we have
	\begin{equation}\label{Intro16}
		\frac{1}{|\emph{supp } f|}	\int_{\R^n} \bigg[\exp \bigg(\frac{|f(x)|}{\beta  \omega_n^{1/r} \mathcal{G}_r(f)} \bigg)^{n'} - 1 \bigg] \, dx \leq A,
	\end{equation}
	where $A$ is an absolute constant.
\end{thm}

The quantity  $\mathcal{G}_r(f)$ is  strongly connected with classical  polar projection body in the following sense.

\begin{thm}\label{Theorem5.6}
	There exists a purely dimensional constant $c_n$ such that\footnote{$r_0$ is an unessential parameter that indicates that $r$ is close to $n$. Without loss of generality, one may think that $r_0 = 2n$.}
	\begin{equation}\label{D2}
		\mathcal{G}_r(f) \leq c_n  |\Pi^*_n f|^{-1/n}, \qquad  \forall r \in (n, r_0),
	\end{equation}
	and
	\begin{equation}\label{LimP1}
	\lim_{r \to n+} \mathcal{G}_r(f) = n^{-\frac{1}{n}}  |\Pi^*_n f|^{-\frac{1}{n}}, \qquad \text{for} \qquad  f \in C^2_c(\R^n).
\end{equation}
\end{thm}

\begin{rem}
We claim that \eqref{LimP1} is the critical version of \eqref{HLLimit} with $p=n$, namely,
\begin{equation}\label{LimP1Special}
		\lim_{s \to 1-}  (1-s)^{\frac{1}{n}} |\Pi^{*, s}_n f|^{-\frac{s}{n}} = n^{-\frac{1}{n}} |\Pi^*_n f|^{-\frac{1}{n}}.
\end{equation}
Observe that both limits \eqref{LimP1} and \eqref{LimP1Special} have the same outcome, but they refer to different differential dimensions. To be more precise, recall that the so-called \emph{differential dimension} of $B^s_{p, q}$ is defined by \cite[Remark 11.5, pp. 172--173]{Triebel01} $$d = d(B^s_{p, q})= s -\frac{n}{p}$$
and the sign of $d$ classifies embedding theorems for $B^{s}_{p, q}$ in three possible regimes: sub-critical $d < 0$, critical $d=0$, and super-critical $d > 0$. In particular $$d(W^{s, n}) = d(B^s_{n, n}) = s-1 < 0 \qquad \text{(sub-critical case)}$$ and $$d\Big(B^{\frac{n}{r}}_{r, n} \Big) = \frac{n}{r}-\frac{n}{r} = 0 \qquad \text{(critical case)}.$$
Hence  \eqref{LimP1} corresponds to the critical version of \eqref{LimP1Special}, where the corresponding asymptotics are studied with respect to the smoothness parameter $s$ for a fixed integrability $n$. This is not the case in \eqref{LimP1}, where the asymptotical relation involves
both  smoothness $n/r$ and integrability $r$ parameters.
\end{rem}
%which involves  with respect to both smoothness and integrability parameters.

Comparing \eqref{Intro16} with \eqref{Intro7} and \eqref{Intro5}, it follows from Theorem \ref{Theorem5.6}  that \eqref{Intro16} provides a significantly stronger estimate than \eqref{Intro7} in both senses: pointwise and asymptotic. In comparison with  \eqref{Intro5} % where only asymptotic improvement (as $s \to 1$) of  \eqref{Intro2} is achieved,
 the pointwise strengthen exhibited by \eqref{Intro16} is a new phenomenon.

\begin{thm}\label{ThmOptim}
\begin{enumerate}
	\item[\rm(i)] (Pointwise improvement) For every $f \in B^{n/r}_{r, n}, r \in (n, r_0)$, with $0 < |\emph{supp }f | < \infty$, we have
		\begin{equation*}
		\frac{1}{|\emph{supp } f|}	 \int_{\R^n} \bigg[ \exp \bigg(c_n' \frac{|f(x)|}{  |\Pi^{*}_{n} f|^{-1/n}} \bigg)^{n'} -1 \bigg] \, dx \leq \frac{1}{|\emph{supp } f|}	 \int_{\R^n} \bigg[\exp \bigg( \frac{|f(x)|}{\beta \omega_n^{1/r} \mathcal{G}_r(f)} \bigg)^{n'} - 1 \bigg] \, dx \leq A,
	\end{equation*}
	where the constants $A, \beta$ are given in Theorem \ref{ThmFractCLYYIntro} and $c_n'$  depends only on $n$.

		\item[\rm(ii)] (Asymptotic improvement) Let $\Omega$ be a bounded open set in $\R^n$.  Then, for every $f \in C^1$ with support in $\Omega$,
	\begin{equation*}
	\lim_{r \to n+}		\int_{\R^n} \bigg[ \exp \bigg(\frac{|f(x)|}{\beta \omega_n^{1/r} \mathcal{G}_r(f)} \bigg)^{n'} -1 \bigg] \, dx =  \int_{\R^n} \bigg[ \exp \bigg(\frac{|f(x)|}{\beta \omega_n^{1/n} n^{-1/n}  |\Pi^{*}_{n} f|^{-1/n}} \bigg)^{n'} -1 \bigg] \, dx.
	\end{equation*}
	\end{enumerate}
\end{thm}

\subsubsection{
\textbf{The BBM phenomenon for fractional Moser--Trudinger inequalities.}} Let $n < r < \infty$ and $q > 1$. It is well known  that there exists a constant $\alpha = \alpha(n, r, q)$, depending only on $n, q$ and $r$, such that (cf. \cite[Theorem 9.1]{Peetre66})
\begin{equation}\label{PMT}
		\frac{1}{|\text{supp } f|} \int_{\R^n} \bigg[ \exp \bigg(\alpha \,  \frac{|f(x)|}{\|f\|_{B^{n/r}_{r, q}}} \bigg)^{q'}-1 \bigg] \, dx \leq A
\end{equation}
for every $f \in B^{n/r}_{r, q}$ with $0< |\text{supp } f| < \infty$. Here, $A$ is an absolute constant.
This inequality can  be viewed as the fractional counterpart of the Moser--Trudinger inequality \eqref{Intro6}.   In sharp contrast to \eqref{Intro6}, the value of the optimal constant $\alpha$ in \eqref{PMT} is unknown, although  some partial results (in the special case $r=q$) have recently been obtained in \cite{PariniRuf} (see also \cite{Martinazzi}).

As an application of the affine inequalities stated in Theorem \ref{ThmFractCLYYIntro} and the fact that (cf. Proposition \ref{PropHolder} and compare with \eqref{RelIntro})
$$
	\mathcal{G}_r(f) \leq c_n  \|f\|_{B^{n/r}_{r, n}},
$$
we are able to get the optimal behaviour of $\alpha$ in \eqref{PMT}   with respect to $r$ showing the Bourgain--Brezis--Mironescu phenomenon for Moser--Trudinger inequalities. More precisely, in Theorem \ref{ThmIntro1.12} below\footnote{This result refers to the case $q=n$ in \eqref{PMT}, but this is not a technical assumption since the method can be easily modified to deal with the general case $q \in [1, \infty)$. For the purposes of this paper, the  case $q=n$ is of special relevance and to pursue the general case would take us too far away from our main motivations.} we establish the family of inequalities \eqref{SharpFMT} that sharpens \eqref{PMT}  and converges as $r \to n+$ to the classical \eqref{Intro6}. In fact, we obtain a stronger assertion: for each $r$, \eqref{SharpFMT} is { pointwisely} stronger than \eqref{Intro6}.

\begin{thm}\label{ThmIntro1.12}
	 Let $f \in B^{n/r}_{r, n}, \, r > n,$ with $0 < |\emph{supp } f| < \infty$. Then there exists a constant $c_{n}$, depending only on $n$, such that
	\begin{equation}\label{SharpFMT}
		\frac{1}{|\emph{supp } f|} \int_{\R^n} \bigg[ \exp \bigg(c_{n} \,  \frac{|f(x)|}{(1-\frac{n}{r})^{1/n} \|f\|_{B^{n/r}_{r, n}}} \bigg)^{n'} - 1 \bigg] \, dx \leq A,
	\end{equation}
	where $A$ is a certain absolute constant.
	Furthermore,  \eqref{SharpFMT}  consists of an improvement of \eqref{Intro6} since
	\begin{equation}\label{SharpFMT2}
	\int_{\R^n} \bigg[ \exp \bigg(\tilde{c}_n \,  \frac{|f(x)|}{\|\nabla f\|_{L^n}} \bigg)^{n'}-1 \bigg] \, dx \leq \int_{\R^n} \bigg[ \exp \bigg(c_n \,  \frac{|f(x)|}{(1-\frac{n}{r})^{1/n} \|f\|_{B^{n/r}_{r, n}}} \bigg)^{n'} - 1 \bigg] \, dx
	\end{equation}
	for every $r \in (n, r_0)$ and it is optimal in the sense that (cf. \eqref{Intro6})
	\begin{equation}\label{SharpFMT1}
		\lim_{r \to n+} \int_{\R^n} \bigg[ \exp \bigg(c_n \,  \frac{|f(x)|}{(1-\frac{n}{r})^{1/n} \|f\|_{B^{n/r}_{r, n}}} \bigg)^{n'} - 1 \bigg] \, dx =  \int_{\R^n} \bigg[ \exp \bigg(\frac{n^{1/n} c_n}{\gamma_n} \,  \frac{|f(x)|}{\|\nabla f\|_{L^n}} \bigg)^{n'}-1 \bigg] \, dx
	\end{equation}
	for every $f \in C^2_c$, where $\gamma_n =\big( \frac{\alpha_{n, n}}{2 n \omega_n} \big)^{1/n}$.
\end{thm}

For the proof of this result see Section \ref{Section6}.

\subsubsection{\textbf{Affine and fractional Moser--Trudinger inequalities: Putting all the pieces together}}
Combining Theorems \ref{ThmFractCLYYIntro}, \ref{ThmOptim} and \ref{ThmIntro1.12}, we complete the picture on affine and fractional Moser-Trudinger inequalities. See  Table \ref{Table2} and compare the outcome with Table \ref{Table1}.

\vspace{5mm}

\begin{table}[H]
  \centering
\begin{tabular}{c|c|c}
  \eqref{Intro6} & $\forall r \in (n, r_0)$ & Theorem \ref{ThmIntro1.12}  \\
  $\frac{1}{|\text{supp } f|} \int_{\R^n} \Big[ \exp \big(C \, \frac{|f(x)|}{\|\nabla f\|_{L^n}} \big)^{n'} - 1 \Big]  \, dx \leq A$ & $\Leftarrow$ & $\frac{1}{|\text{supp } f|} \int_{\R^n} \Big[ \exp \big(c_{n} \,  \frac{|f(x)|}{(1-\frac{n}{r})^{1/n} \|f\|_{B^{n/r}_{r, n}}} \big)^{n'} - 1 \Big] \, dx \leq A$ \\\hline
  $\Uparrow$ & \ & $\Uparrow$ \\\hline
  \eqref{Intro7} & $\forall r \in (n, r_0)$  & Theorem \ref{ThmFractCLYYIntro} \\
  $\frac{1}{|\text{supp } f|} \int_{\R^n}  \Big[ \exp \big(C \, \frac{|f(x)|}{c_n |\Pi^*_n f|^{-1/n}} \big)^{n'} - 1\Big] \, dx  \leq A$ & $\Leftarrow$ &  $\frac{1}{|\text{supp } f|}	 \int_{\R^n} \Big[\exp \big( \frac{|f(x)|}{\beta \omega_n^{1/r} \mathcal{G}_r(f)} \big)^{n'} - 1 \Big] \, dx \leq A$
\end{tabular}
\vspace{2mm}
  \caption{Affine and fractional Moser--Trudinger inequalities.}
    \label{Table2}
\end{table}

\subsubsection{\textbf{Affine and fractional Morrey inequalities}} In the super-limiting regime $s > n/p$,
 we obtain in  Section \ref{Section7} the following result.

\begin{thm}\label{ThmIntro116}
	Let $p > n$ and $s_0 \in (\frac{n}{p}, 1)$.  If\footnote{$s_0$ plays an auxiliary role and only indicates that values of $s$ sufficiently close to $1$ are of some interest.} $s \in (s_0, 1)$ then there exists a constant $C$, which is independent of $s$,  such that
	\begin{equation}\label{ThmPam1}
		\|f\|_{L^\infty} \leq C   |\emph{supp } f|^{\frac{s}{n}-\frac{1}{p}} (1-s)^{\frac{1}{p}}  |\Pi^{*, s}_p f|^{-\frac{s}{n}}
	\end{equation}
	provided that $f \in W^{s, p}$ with $|\emph{supp } f| < \infty$. Furthermore
	\begin{equation}\label{ThmPam2}
		(s(1-s))^{\frac{1}{p}}  |\emph{supp } f|^{\frac{s}{n}}  |\Pi^{*, s}_p f|^{-\frac{s}{n}} \leq c  |\emph{supp } f|^{\frac{1}{n}}  |\Pi^*_p f|^{-\frac{1}{n}}, \qquad \forall s \in (0, 1),
	\end{equation}
	 and
	\begin{equation}\label{ThmPam3}
		\lim_{s \to 1-}  (1-s)^{\frac{1}{p}}  |\emph{supp } f|^{\frac{s}{n}-\frac{1}{p}} |\Pi^{*, s}_p f|^{-\frac{s}{n}}= p^{-\frac{1}{p}}  |\emph{supp } f|^{\frac{1}{n}-\frac{1}{p}}  |\Pi^*_p f|^{-\frac{1}{n}}.
	\end{equation}
\end{thm}

 It is worth mentioning that, unlike the case $s= n/p$, the full strength provided by Besov polar projection bodies is not now needed to achieve Theorem \ref{ThmIntro116}, which is formulated in terms of fractional polar projection bodies.  In particular, a key role in our argument is played by the following sharp form of the Morrey inequality in terms of Gagliardo norms: under the assumptions\footnote{We mention that \eqref{Intro121} with $s \to (n/p)+$ was already investigated in \cite{Triebel05}.} of \eqref{ThmPam1}, we have
\begin{equation}\label{Intro121}
	\|f\|_{L^\infty} \leq C   |\text{supp } f|^{\frac{s}{n}-\frac{1}{p}} (1-s)^{\frac{1}{p}} \|f\|_{W^{s, p}}.
\end{equation}
This inequality seems to be new and it is interesting by its own sake. Its proof is also contained in Section \ref{Section7}. Note that   \eqref{Intro8} follows immediately from \eqref{Intro121} via the classical BBM formula \eqref{Intro3}.

\vspace{5mm}

\begin{table}[H]
  \centering
\begin{tabular}{c|c|c}
  \eqref{Intro8} & $1 \leftarrow s$ & \eqref{Intro121}  \\
  $\|f\|_{L^\infty} \leq c |\text{supp f}|^{\frac{1}{n}-\frac{1}{p}} \|\nabla f\|_{L^p}$ & $\Leftarrow$ & $\|f\|_{L^\infty} \leq C   |\text{supp } f|^{\frac{s}{n}-\frac{1}{p}} (1-s)^{\frac{1}{p}} \|f\|_{W^{s, p}}$ \\\hline
  $\Uparrow$ & \ & $\Uparrow$ \\\hline
  \eqref{Intro9} & $1 \leftarrow s$  & Theorem \ref{ThmIntro116} \\
  $\|f\|_{L^\infty} \leq c |\text{supp f}|^{\frac{1}{n}-\frac{1}{p}} c_{n, p} |\Pi^*_p f|^{-1/n}$ & $\Leftarrow$ &  $\|f\|_{L^\infty} \leq C   |\text{supp } f|^{\frac{s}{n}-\frac{1}{p}} (1-s)^{\frac{1}{p}}  |\Pi^{*, s}_p f|^{-\frac{s}{n}}$
\end{tabular}
\vspace{2mm}

  \caption{Affine and fractional Morrey inequalities.}
  \label{Table3}
\end{table}

\subsubsection{\textbf{A few more words about our techniques}}

In addition to the concept of Besov polar projection bodies mentioned above, our approach relies on extrapolation estimates (cf. Theorem \ref{ThmEmbBesov1}). These estimates are based on a careful analysis of decomposition methods using wavelets and interpolation tools ($K$-functionals); see Section \ref{Section4}.
Moreover,  we also establish new  Poincar\'e inequalities (cf. Theorem \ref{ThmOptim2}) that improve the recent results by Haddad, Jim\'enez, and Montenegro \cite{HaddadJimenezMontenegro}.

%Apart from the above mentioned concept of Besov polar projection bodies, our approach relies on {\bf a variety of techniques} that are developed in detail in Section \ref{Section4}. Among them, we establish extrapolation estimates (cf. Theorem \ref{ThmEmbBesov1}) that are based on a careful analysis of decomposition methods via wavelets and interpolation tools ($K$-functionals). Moreover,  we also establish new  Poincar\'e inequalities (cf. Theorem \ref{ThmOptim2}) that improve the recent results by Haddad, Jim\'enez, and Montenegro \cite{HaddadJimenezMontenegro}.

\section{Background}
\subsection{Symmetrization}\label{Symm}

Let $E \subset \R^n$ be a Borel set of finite measure. The \emph{Schwarz symmetral} $E^\star$ of $E$ is the closed centered Euclidean ball with $|E^\star| = |E|$.

Let $f: \R^n \to [0, \infty)$ be a measurable function with corresponding level sets $\{f \geq t\}$ of finite measure for every $t > 0$. Note that $f$ can be recovered from $\{f \geq t\}$ via the so-called \emph{layer-cake formula}:
$$
	f(x) = \int_0^\infty \mathbf{1}_{\{f \geq t\}}(x) \, d t,
$$
where $\mathbf{1}_E$ denotes the indicator function of the measurable set $E$. This representation can be applied to introduce the concept of Schwarz symmetral: Let $f : \R^n \to \R$ be a measurable function such that $\{|f| \geq t\}$ has finite measure for every $t > 0$. Then  $f^\star$, the \emph{Schwarz symmetral} of $f$, can be introduced as
$$
	f^\star(x) =  \int_0^\infty \mathbf{1}_{\{|f| \geq t\}^\star} (x) \, dt, \qquad x \in \R^n.
$$
Note that $f^\star$ is a radially decreasing functions with $|\{|f| \geq t\}| = |\{f^\star \geq t\}|$. The function $f^\star$ is also commonly known as spherically symmetric rearrangement.

\subsection{Star bodies}
Let $K$ be a star-shaped (with respect to the origin) set in $\R^n$ with corresponding gauge function defined by
\begin{equation*}
	\|x\|_K = \inf \{\lambda > 0 : x \in \lambda K\}, \qquad x \in \R^n.
\end{equation*}
It is plain to see that $|K|$, the $n$-dimensional Lebesgue measure of $K$, can be expressed as
\begin{equation}\label{Volume}
	|K| = \frac{1}{n} \int_{\S^{n-1}} \rho_K(\xi)^n \, d \xi,
\end{equation}
where\footnote{We assume that $\rho_K$ is a measurable function.} $\rho_K$ is defined by
$$
	\rho_K(x) = \|x\|_K^{-1} = \sup \{\lambda \geq 0 : \lambda x \in K\}, \qquad x \in \R^n \backslash \{0\}.
$$
The star-shaped set $K$ is said to be a \emph{star body} if $\rho_K$ is a strictly positive and continuous function on $\R^n \backslash \{0\}$.

Let $\alpha \in \R \backslash \{0, n\}$. For star-shaped sets $K$ and $L$, the \emph{dual mixed volume} is defined by
$$
	\tilde{V}_\alpha (K, L) = \frac{1}{n} \int_{\S^{n-1}} \rho_K(\xi)^{n-\alpha} \rho_L(\xi)^\alpha \, d \xi.
$$

\subsection{$L^p$ and fractional polar projection bodies}\label{SectionPrelimPolar}
The concept of $L^p$ \emph{polar projection body} associated to $f \in W^1_p, \, p \geq 1,$ goes back to the work of Lutwak, Yang,  and Zhang \cite{LutwakYangZhang} and it is defined as the star body with gauge function given by
\begin{equation*}
	\|\xi\|_{\Pi^*_p f}^p = \int_{\R^n} |\nabla f (x) \cdot \xi|^p \, dx,  \qquad \xi \in \R^n.
\end{equation*}
In particular, $\Pi^*_p f$ the polar body of a convex body.

Fractional counterparts of $\Pi^*_p f$ have been recently introduced by Haddad and Ludwig \cite{HaddadLudwig}. More precisely, given $f \in W^{s, p}, s \in (0, 1), p \in [1, \infty)$,  its \emph{fractional polar projection body}  $\Pi^{*, s}_p f$ is defined as the star-shaped set with corresponding gauge function:
	\begin{equation}\label{FPPB}
		\|\xi\|_{\Pi^{*, s}_p f}^{s p} =  \int_0^\infty t^{-s p} \int_{\R^n} |f(x + t \xi)-f(x)|^p \, dx  \, \frac{dt}{t}, \qquad \xi \in \R^n.
	\end{equation}
	We refer to \cite[Proposition 4]{HaddadLudwig} for some basic properties of $\Pi^{*, s}_p f$.

\section{Anisotropic moduli of smoothness and Besov polar projection bodies}\label{Section3}

\subsection{Anisotropic moduli of smoothness: definition and basic properties}\label{SectionMod}

\begin{defn}\label{DefModSec}
Let $K$ be a star body in $\R^n$. The \emph{anisotropic moduli of smoothness} of  $f \in L^p, 1 \leq p \leq \infty,$ is given by
\begin{equation}\label{DefMod}
	\omega_K (f, t)_p =\bigg(\frac{1}{t^n |K|} \int_{\|h\|_K < t} \|\Delta_h f\|^p_{L^p} \, d h \bigg)^{\frac{1}{p}}, \qquad t > 0,
\end{equation}
(with the usual interpretation if $p=\infty$).
Here, $\Delta_h$ is the difference operator, i.e., $\Delta_h f (x) = f(x+h)-f(x)$.
\end{defn}

\begin{rem}\label{Rem37}
(i) Note that $\omega_K(f, t)_p \leq 2 \|f\|_{L^p}$ for every $f \in L^p$ and $t > 0$.

(ii) Let $K = B^n$, the unit ball in $\R^n$. Then $\omega_K (f, t)_p = \omega(f, t)_p$, the classical (averaged)  moduli of smoothness (see e.g. \cite{DeVoreLorentz}):
	\begin{equation}\label{CMod}
		\omega(f, t)_p = \bigg(\frac{1}{t^n \omega_n} \int_{|h| < t} \|\Delta_h f\|^p_{L^p} \, d h \bigg)^{\frac{1}{p}},
	\end{equation}
	where $|h|$ is the Euclidean norm of $h$.
\end{rem}

Next we study some basic properties of $\omega_K(f, t)_p$. In particular, we show that the $L^p$-average over $\|h\|_K < t$ in \eqref{DefMod} can be replaced  by any $L^q$-average, more precisely, we obtain the following result.

\begin{prop}\label{LemmaTechnical}
Given $p \in [1, \infty)$ and $q, r \in [1, \infty]$, let
$$
	g_q(t) = \bigg(\frac{1}{t^n |K|} \int_{\|h\|_K < t} \|\Delta_h f\|^q_{L^p} \, d h \bigg)^{\frac{1}{q}}.
$$
Then
	\begin{equation}\label{Lem1}
	g_r(t)	 \leq 4 (2^n + 1)^{\frac{1}{q}} g_q(t).
	\end{equation}
	Furthermore, the constant in \eqref{Lem1} can be improved to
	\begin{equation}\label{Lem2}
		g_r(t) \leq g_q(t)
	\end{equation}
	 if, additionally, $r \leq q$. In particular
	 \begin{equation}\label{EquivMod}
	 \frac{1}{4 (2^n + 1)^{\frac{1}{p}}} \, \sup_{\|h\|_K < t} \|\Delta_h f\|_{L^p}	\leq \omega_K(f, t)_p \leq \sup_{\|h\|_K < t} \|\Delta_h f\|_{L^p}.
	 \end{equation}
\end{prop}

\begin{rem}
	The equivalence \eqref{EquivMod} is well known in the classical setting $K = B^n$  (cf. \eqref{CMod}), that is\footnote{Given two non-negative quantities $A$ and $B$, by $A \lesssim B$ we mean that there exists a constant $C$, independent of all essential parameters, such that $A \leq C B$. We write $A \approx B$ if $A \lesssim B \lesssim A$.},
	$$
		\omega(f, t)_p \approx \sup_{|h| < t} \|\Delta_h f\|_{L^p}.
	$$
\end{rem}

\begin{proof}[Proof of Proposition \ref{LemmaTechnical}]
For any $h, x, \xi \in \R^n$,
	$$
		\Delta_h f (x) = - (\Delta_{\xi-h} f(x+h) - \Delta_\xi f (x)),
	$$
	which yields (using also that $\|\Delta_{\xi-h} f\|_{L^p} \leq 2 \|\Delta_{\frac{\xi-h}{2}} f\|_{L^p}$)
	$$
		\|\Delta_h f\|^q_{L^p} \leq 4^{ q}  ( \|\Delta_{\frac{\xi-h}{2}} f \|_{L^p}^q + \|\Delta_\xi f\|^q_{L^p}).
	$$
	Assume $\|h\|_K < t$. Then integrating the last inequality over all $\|\xi\|_K < t$, we get
	\begin{align*}
		t^n |K| \|\Delta_h f\|^q_{L^p} &\leq 4^q \bigg(  \int_{\|\xi\|_K < t}   \|\Delta_{\frac{\xi-h}{2}} f \|^q_{L^p} \, d \xi  +  \int_{\|\xi\|_K < t}   \|\Delta_{\xi} f \|^q_{L^p} \, d \xi \bigg) \\
		&\leq 4^q \bigg(  \int_{\|\frac{\xi-h}{2}\|_K < t}   \|\Delta_{\frac{\xi-h}{2}} f \|^q_{L^p} \, d \xi  +  \int_{\|\xi\|_K < t}   \|\Delta_{\xi} f \|^q_{L^p} \, d \xi \bigg) \\
		& = 4^q (2^n + 1) \int_{\|\xi\|_K < t}   \|\Delta_{\xi} f \|^q_{L^p} \, d \xi.
	\end{align*}
	Hence
	$$
		\|\Delta_h f\|_{L^p} \leq  4 (2^n + 1)^{1/q} \bigg( \frac{1}{|K| t^n} \int_{\|\xi\|_K < t}   \|\Delta_{\xi} f \|^q_{L^p} \, d \xi \bigg)^{1/q}
	$$
	and integrating over all $\|h\|_K < t$, we arrive at
	$$
		\frac{1}{|K| t^n}\int_{\|h\|_K < t} \|\Delta_h f\|_{L^p}^r \, dh\leq 4^r (2^n + 1)^{r/q}  \bigg( \frac{1}{|K| t^n} \int_{\|\xi\|_K < t}   \|\Delta_{\xi} f \|^q_{L^p} \, d \xi \bigg)^{r/q}.
	$$
	This completes the proof of \eqref{Lem1}.
	
		On the other hand, the proof of \eqref{Lem2} is an application of H\"older's inequality:
	\begin{align*}
		 \int_{\|h\|_K < t} \|\Delta_h f\|^r_{L^p} \, d h  & \leq \bigg(\int_{\|h\|_K < t} \|\Delta_h f\|_{L^p(\R^n)}^q \, dh \bigg)^{\frac{r}{q}} \bigg( \int_{\|h\|_K < t} \, dh \bigg)^{1-\frac{r}{q}}  \\
		 & =|K| t^n  \bigg(\frac{1}{|K| t^n} \int_{\|h\|_K < t} \|\Delta_h f\|_{L^p}^q \, dh \bigg)^{\frac{r}{q}}.
	\end{align*}
\end{proof}

\begin{prop}
	\begin{enumerate}
	\item[\rm(i)] $\omega_K(f, t)_p$ is an almost increasing function of $t$, that is,  $\omega_K(f, t)_p \lesssim \omega_K(f, u)_p$ for every $t < u$.
	\item[\rm(ii)]  $\frac{\omega_K(f, t)_p}{t}$ is an almost decreasing function of $t$, that is,  $\frac{\omega_K(f, u)_p}{u} \lesssim  \frac{\omega_K(f, t)_p}{t}$ for every $t < u$.
	\end{enumerate}
\end{prop}

\begin{proof}
	Property (i) is an immediate consequence of \eqref{EquivMod}. On the other hand, we claim that
	\begin{equation}\label{MonMod1}
		\omega_K(f, \lambda t)_p \lesssim (1 + \lambda) \, \omega_K(f, t)_p
	\end{equation}
	for every $\lambda, t > 0$. Observe that (ii) follows from \eqref{MonMod1}. Since
	$$
		\Delta_{N h} f (x) = \sum_{k=0}^{N-1} \Delta_h f (x + k h), \qquad N \in \N,
	$$
	we have $\|\Delta_{N h} f\|_{L^p} \leq N \|\Delta_h f \|_{L^p}$. Then, by a simple change of variables,
	\begin{align}
	 \omega_K(f, N t)_p &= \bigg(\frac{1}{ t^n |K|} \int_{\|h\|_K < t} \|\Delta_{N h} f\|^p_{L^p} \, d h \bigg)^{\frac{1}{p}} \nonumber \\
	 & \leq N  \bigg(\frac{1}{ t^n |K|} \int_{\|h\|_K < t} \|\Delta_{h} f\|^p_{L^p} \, d h \bigg)^{\frac{1}{p}} = N \omega_K(f, t)_p.  \label{AuxEs}
	\end{align}
	As a consequence, using also property (i),
	$$
		\omega_K(f, \lambda t)_p \lesssim \omega_K(f, (1 + [\lambda]) t )_p \leq (1 + [\lambda ]) \, \omega_K(f, t)_p \leq (1 + \lambda) \, \omega_K(f, t)_p,
	$$
	i.e., \eqref{MonMod1} holds.
\end{proof}

In the isotropic setting given by $K = B^n$,  $\omega(f, t)_p$ (cf. \eqref{CMod}) allows to measure smoothness properties of $f$. Indeed,  the Hardy--Littlewood theorem asserts that, for $p \in (1, \infty)$,
\begin{equation}\label{HLEquiv}
	\omega(f, t)_p \lesssim t \iff \int_{\R^n} |\nabla f (x)|^p \, dx < \infty,
\end{equation}
cf. \cite[Theorem 9.3, p. 53]{DeVoreLorentz} or \cite[Proposition 3, p. 139]{Stein}. Our next result extends \eqref{HLEquiv} to general convex bodies. At a first sight, one may conjecture that $\omega_K(f, t)_p$ characterizes the anisotropic Sobolev space (see e.g. \cite{Cordero, Figalli})
$$
	\|f\|_{W^{1, p}_K} = \bigg(\int_{\R^n} \|\nabla f (x)\|^p_{K^*} \, dx \bigg)^{\frac{1}{p}},
$$
where $K^* = \{y \in \R^n: y \cdot x \leq 1 \text{ for all } x \in K\}$, the polar body of $K$. However, we show that this is not true in general, but the correct answer involves instead variants of $W^{1, p}_K$ as introduced by Ludwig \cite{Ludwigb, Ludwig}. More precisely, anisotropic Sobolev spaces based on the norm with unit ball $Z_p^* K$, the polar $L^p$ moment   body of $K$. Recall that, for a convex body $K$, the polar $L^p$ moment body of $K$ is the unit ball of the norm defined by
$$
	\|x\|^p_{Z_p^* K} = \frac{n + p}{2} \, \int_K |x \cdot y|^p \, dy, \qquad x \in \R^n.
$$

\begin{prop}[Anisotropic Hardy--Littlewood theorem]\label{Prop3.13}
	Let $p \in [1, \infty)$ and $f \in W^{2, p}(\R^n)$ with compact support. Then
	$$
		\omega_K(f, t)_p \lesssim t \iff \int_{\R^n} \|\nabla f (x)\|^p_{Z_p^* K} \, d x < \infty.
	$$
	Furthermore
	$$
		\lim_{t \to 0}  \frac{\omega_K(f, t)_p}{t} =  \sup_{t > 0} \frac{\omega_K(f, t)_p}{t} \approx \bigg( \int_{\R^n} \|\nabla f (x)\|^p_{Z_p^* K} \, d x \bigg)^{1/p}.
	$$
\end{prop}

\begin{proof}
Recall that, for each $ h \in \R^n$ and for almost every $x \in \R^n$,
$$
	f(x+h)-f(x) = \int_0^1 \nabla f(x + \lambda h) \cdot h \, d \lambda.
$$
Then, by Minkowski's inequality and a change of variables,
\begin{equation}
	\|\Delta_h f\|_{L^p} \leq \int_0^1 \bigg( \int_{\R^n} |\nabla f(x + \lambda h) \cdot h|^p \, dx  \bigg)^{1/p} \, d\lambda  =  \bigg(\int_{\R^n} |\nabla f(x) \cdot h|^p \, d x \bigg)^{1/p}. \label{314b2}
\end{equation}
Using Fubini's theorem and another change of variables yield
\begin{align*}
	 \int_{\|h\|_K < t} \|\Delta_h f\|^p_{L^p} \, d h & \leq \int_{\|h\|_K < t} \int_{\R^n} |\nabla f(x) \cdot h|^p \, d x  \, dh \\
	 & = t^p \int_{\R^n} \int_{\|h\|_K < t} \Big|\nabla f(x) \cdot \frac{h}{t} \Big|^p \, dh \, dx \\
	 & = t^{p+n} \int_{\R^n} \int_{\|h\|_K < 1} |\nabla f (x) \cdot h|^p \, dh \, dx.
\end{align*}
Accordingly (note that $h \in K \iff \|h\|_K < 1$)
$$
	\frac{1}{t^n}  \int_{\|h\|_K < t} \|\Delta_h f\|^p_{L^p} \, d h \leq \frac{2}{n+p} \,  t^p \int_{\R^n} \|\nabla f(x)\|^p_{Z^*_p K} \, dx.
$$
As a consequence
\begin{equation}\label{Estim1}
	\omega_K(f, t)_p \lesssim t \bigg(  \int_{\R^n} \|\nabla f(x)\|^p_{Z^*_p K} \, dx \bigg)^{1/p}.
\end{equation}

Next we prove the converse inequality. Assume that $f \in W^{2, p}$ with compact support.  Let $h \in \R^n$ and $\Omega_f = \text{supp } f$. By the triangle inequality, we have
$$
	\|\nabla f \cdot h\|_{L^p(\Omega_f)} \leq \|\Delta_h f -\nabla f \cdot h\|_{L^p(\Omega_f)} + \|\Delta_h f\|_{L^p(\Omega_f)}.
$$
If we integrate this expression over all $\|h\|_K < t$, we get
\begin{equation}\label{11.1}
	\int_{\|h\|_K < t} \|\nabla f  \cdot h\|_{L^p(\Omega_f)}^p \, dh \lesssim \int_{\|h\|_K < t}  \|\Delta_h f -\nabla f \cdot h\|_{L^p(\Omega_f)}^p \, dh +  \int_{\|h\|_K < t}   \|\Delta_h f\|_{L^p(\Omega_f)}^p \, dh.
\end{equation}
Observe that
$$
	\int_{\|h\|_K < t} \|\nabla f  \cdot h\|_{L^p(\Omega_f)}^p \, dh = t^{p+n} \int_{K} \|\nabla f  \cdot h\|_{L^p(\Omega_f)}^p \, dh \approx t^{p+n} \int_{\Omega_f} \|\nabla f(x)\|^p_{Z_p^* K} \, dx
$$
and then \eqref{11.1} reads as
\begin{align}
	 \int_{\Omega_f} \|\nabla f(x)\|^p_{Z_p^* K} \, dx &\lesssim \frac{1}{t^{p+n}}  \int_{\|h\|_K < t}  \|\Delta_h f -\nabla f \cdot h\|_{L^p(\Omega_f)}^p \, dh + \frac{1}{t^{p+n}} \int_{\|h\|_K < t}   \|\Delta_h f\|_{L^p(\Omega_f)}^p \, dh \nonumber \\
	 & \lesssim \frac{1}{t^{p+n}}  \int_{\|h\|_K < t}  \|\Delta_h f -\nabla f \cdot h\|_{L^p(\Omega_f)}^p \, dh + \frac{\omega_K(f, t)_p^p}{t^p}. \label{11.2}
\end{align}

Next we show that
\begin{equation}\label{11.3n}
	\lim_{t \to 0}  \frac{1}{t^{p+n}}  \int_{\|h\|_K < t}  \|\Delta_h f -\nabla f \cdot h\|_{L^p(\Omega_f)}^p \, dh  = 0.
\end{equation}
Indeed, it is well known that there exists a positive constant $c$, which depends only on $K$ and $n$, such that $c |h| \leq \|h\|_K$ for all $h \in \R^n$. In particular
\begin{equation}\label{11.3}
	 \int_{\|h\|_K < t}  \|\Delta_h f -\nabla f \cdot h\|_{L^p(\Omega_f)}^p \, dh \leq  \int_{|h| < t/c}  \|\Delta_h f -\nabla f \cdot h\|_{L^p(\Omega_f)}^p \, dh.
\end{equation}
Furthermore, we make use of the basic fact that given $x, h \in \R^n$ we can write $f(x+h)-f(x) = \nabla f (\xi) \cdot h$ for some  $\xi_{x, h} \in \R^n$ between $x$ and $x+h$.  Hence, letting $A$  the Lipschitz constant of $\nabla f$,
\begin{align}
	\|\Delta_h f -\nabla f \cdot h\|_{L^p(\Omega_f)}^p &= \int_{\Omega_f} |(\nabla f(\xi_{x, h})  -\nabla f (x)) \cdot h|^p \, d x \nonumber \\
	& \leq |h|^p \int_{\Omega_f} |\nabla f(\xi_{x, h}) - \nabla f (x)|^p  \, dx \nonumber  \\
	& \leq A^p |h|^p \int_{\Omega_f} |\xi_{x, h} - x|^p \, dx  \leq A^p |h|^{2 p} |\Omega_f|. \label{LipEstim}
\end{align}
As a byproduct, by \eqref{11.3},
\begin{equation*}
	 \int_{\|h\|_K < t}  \|\Delta_h f -\nabla f \cdot h\|_{L^p(\Omega_f)}^p \, dh \leq A^p  |\Omega_f|  \int_{|h| < t/c}  |h|^{2 p}\, dh  = k_n  A^p |\Omega_f| \Big(\frac{t}{c} \Big)^{2 p + n},
\end{equation*}
which leads to
$$
	 \frac{1}{t^{p+n}}  \int_{\|h\|_K < t}  \|\Delta_h f -\nabla f \cdot h\|_{L^p(\Omega_f)}^p \, dh  \leq k_n  A^p |\Omega_f| c^{-2 p-n} t^p \to 0 \qquad \text{as} \qquad t \to 0.
$$
This proves the desired assertion \eqref{11.3n}.

Taking limits as $t \to 0$ in \eqref{11.2}, it follows from \eqref{11.3n} that
\begin{equation}\label{Estim2}
	 \int_{\Omega_f} \|\nabla f(x)\|^p_{Z_p^* K} \, dx  \lesssim \limsup_{t \to 0}  \frac{\omega_K(f, t)_p^p}{t^p} \leq \sup_{t > 0}   \frac{\omega_K(f, t)_p^p}{t^p}.
\end{equation}
Putting together \eqref{Estim1} and \eqref{Estim2}, we get
$$
	\sup_{t > 0} \frac{\omega_K(f, t)_p}{t} \approx \bigg( \int_{\Omega_f} \|\nabla f(x)\|^p_{Z_p^* K} \, dx  \bigg)^{\frac{1}{p}}.
$$

It remains to show that
\begin{equation}\label{ClaimN}
	\lim_{t \to 0}  \frac{\omega_K(f, t)_p}{t} =  \sup_{t > 0} \frac{\omega_K(f, t)_p}{t}.
\end{equation}
Let
$$
	L = \liminf_{t \to 0}  \frac{\omega_K(f, t)_p}{t}
$$
and choose a sequence of positive numbers $\{t_j\}_{j \in \N}$ such that $t_j \to 0$ as $j \to \infty$ and
\begin{equation}\label{AuxEs2}
	\lim_{j \to \infty} \frac{\omega_K(f, t_j)_p}{t_j} =  L.
\end{equation}
For any $t > 0$, we let $k_j \in \N$ be such that $k_j t_j \leq t < (k_j + 1) t_j$. Note that $\lim_{j \to \infty} k_j t_j = t$. Hence, taking into account the continuity of $\omega_K(f, t)$ as a function of $t$ and using \eqref{AuxEs} and \eqref{AuxEs2},  we have
$$
	  \frac{\omega_K(f, t)_p}{t} = \lim_{j \to \infty} \frac{\omega_K(f, k_j t_j)_p}{k_j t_j} \leq  \lim_{j \to \infty} \frac{\omega_K(f,  t_j)_p}{ t_j}  = L.
$$
This implies $\sup_{t > 0}  \frac{\omega_K(f, t)_p}{t} \leq L$ and the proof of \eqref{ClaimN} is complete.
\end{proof}

An important result in classical theory of function spaces claims that moduli of smoothness does not increase under  symmetrizations. Namely, for every $t >0$,
	\begin{equation}\label{PointEstimRea}
		\omega(f^\star, t)_p \leq \omega(f, t)_p,
	\end{equation}
	where $\omega(f, t)_p$ is given by \eqref{CMod}. See \cite{Kolyada} and the extensive list of  references given there. Next we extend this result from $K = B^n$ to the general anisotropic setting.

%Next we obtain a quantitative version of the anisotropic  P\'olya-Szeg\H{o} inequality.

\begin{prop}\label{ThmPS}
Assume that  $f \in L^p, 1 \leq p < \infty$. Then, for each $t >0$,
	\begin{equation}\label{ForPS}
		\omega_{K^\star}(f^\star, t)_p \leq \omega_K(f, t)_p.
	\end{equation}
\end{prop}

\begin{proof}
We first observe that it is enough to show the validity of \eqref{ForPS} for non-negative functions.  Indeed, the case of general functions $f$ can be reduced to the latter since $\omega_K(|f|, t)_p \leq \omega_K(f, t)_p$. Then we restrict our attention to  non-negative functions $f$.

	By Fubini's theorem, we have
	\begin{align}
		 \int_{\|h\|_K < t} \|\Delta_h f\|^p_{L^p} \, d h &=  \int_{\|h\|_K < t} \int_{\R^n} |f(x+h)-f(x)|^p \, dx \, dh \nonumber \\
		 & = \int_{\R^n}  \int_{\|x-y\|_K < t}  |f(y)-f(x)|^p \, d y \, dx  \nonumber\\
		 & = \int_{\R^n} \int_{\R^n}  |f(y)-f(x)|^p \mathbf{1}_{Z_t}(x-y) \, dy \, dx, \label{PS1}
	\end{align}
	where $Z_t = \{z \in \R^n : \|z\|_K < t\}$. Furthermore, we find that
	\begin{equation}\label{PS2}
		z \in Z_t \iff z \in t K.
	\end{equation}
	Indeed, it is clear that $z \in tK$ yields $\|z\|_K \leq t$. Conversely, if $\|z\|_K < t$ then there exists $\lambda < t$ such that $z \in \lambda K$, i.e., $z = \lambda k$ for some $k \in K$. In particular, $z = t \frac{z}{t} = t \frac{\lambda k}{t}$ and $\frac{\lambda k}{t} \in K$, since $\lambda/t \in (0, 1)$ and $k \in K$ (recall that $K$ is a star body).
	
	In view of \eqref{PS2}, one can rewrite \eqref{PS1} as
	\begin{equation}\label{PS3}
		 \int_{\|h\|_K < t} \|\Delta_h f\|^p_{L^p} \, d h  = \int_{\R^n} \int_{\R^n}  |f(x)-f(y)|^p \mathbf{1}_{t K}(x-y) \, dy \, dx.
	\end{equation}
	
	Taking into account that $f$ is non-negative, one can write\footnote{Recall that $f_+ = \max \{f, 0\}$ and $f_{-} =  -\min\{f, 0\}$.}
	$$
		(f(x)-f(y))_+^p = p \int_0^\infty (f(x)-r)_+^{p-1} \mathbf{1}_{\{f < r\}}(y) \, dr.
	$$
	The case $p=1$ in the previous formula should be adequately interpreted as $(f(x)-r)_+^{p-1} = \mathbf{1}_{\{f \geq r\}}(x)$. Then, by Fubini's theorem,
	\begin{align}
		\int_{\R^n} \int_{\R^n}  (f(x)-f(y))_+^p \mathbf{1}_{t K}(x-y) \, dy \, dx &= p  \int_0^\infty \int_{\R^n} \int_{\R^n} (f(x)-r)_+^{p-1} \mathbf{1}_{\{f < r\}}(y)  \mathbf{1}_{t K}(x-y) \, dy \, dx \, dr \nonumber \\
		&\hspace{-4cm}= p  \int_0^\infty \int_{\R^n} \int_{\R^n} (f(x)-r)_+^{p-1} (1- \mathbf{1}_{\{f \geq r\}}(y))  \mathbf{1}_{t K}(x-y) \, dy \, dx \, dr. \label{326new}
	\end{align}
	For every $r, t > 0$, we have
	\begin{align}
		\int_{\R^n} \int_{\R^n} (f(x)-r)_+^{p-1} (1- \mathbf{1}_{\{f \geq r\}}(y))  \mathbf{1}_{t K}(x-y) \, dy \, dx & \nonumber \\
		& \hspace{-7cm} = \int_{\R^n} (f(x)-r)_+^{p-1} \int_{\R^n}   \mathbf{1}_{t K}(x-y) \, dy \, dx  - \int_{\R^n} \int_{\R^n} (f(x)-r)_+^{p-1} \mathbf{1}_{\{f \geq r\}}(y)  \mathbf{1}_{t K}(x-y) \, dy \, dx \nonumber \\
		&  \hspace{-7cm} = t^n |K|  \int_{\R^n} (f(x)-r)_+^{p-1} \, dx  - \int_{\R^n} \int_{\R^n} (f(x)-r)_+^{p-1} \mathbf{1}_{\{f \geq r\}}(y)  \mathbf{1}_{t K}(x-y) \, dy \, dx, \label{326}
	\end{align}
	where the last step is justified by the fact that   $\int_{\R^n} (f(x)-r)_+^{p-1} \, dx < \infty$. The latter follows since $f \in L^p$ and therefore $\{f > r\}$ has finite measure. Moreover, applying  symmetrization properties, we can write
	$$|K|  \int_{\R^n} (f(x)-r)_+^{p-1} \, dx  =  |K^\star|  \int_{\R^n} (f^\star(x)-r)_+^{p-1} \, dx.$$
	
	On the other hand, the second term in \eqref{326} can be estimated using Riesz's rearrangement inequality. Indeed
	\begin{equation*}
		\int_{\R^n} \int_{\R^n} (f(x)-r)_+^{p-1} \mathbf{1}_{\{f \geq r\}}(y)  \mathbf{1}_{t K}(x-y) \, dy \, dx  \leq \int_{\R^n} \int_{\R^n} (f^\star(x)-r)_+^{p-1} \mathbf{1}_{\{f^\star \geq r\}} (y) \mathbf{1}_{t K^\star}(x-y) \, dy \, dx.
	\end{equation*}
	Accordingly
	\begin{align*}
		\int_{\R^n} \int_{\R^n} (f(x)-r)_+^{p-1} (1- \mathbf{1}_{\{f \geq r\}}(y))  \mathbf{1}_{t K}(x-y) \, dy \, dx \\
		& \hspace{-7cm} \geq t^n |K^\star|  \int_{\R^n} (f^\star(x)-r)_+^{p-1} \, dx  - \int_{\R^n} \int_{\R^n} (f^\star(x)-r)_+^{p-1} \mathbf{1}_{\{f^\star \geq r\}} (y) \mathbf{1}_{t K^\star}(x-y) \, dy \, dx \\
		& \hspace{-7cm} = \int_{\R^n} \int_{\R^n} (f^\star(x)-r)_+^{p-1} (1- \mathbf{1}_{\{f^\star \geq r\}}(y))  \mathbf{1}_{t K^\star}(x-y) \, dy \, dx,
	\end{align*}
	where the last step follows from \eqref{326} (after replacing $f$ and $K$ by $f^\star$ and $K^\star$, respectively). Therefore \eqref{326new} can be estimated as
	\begin{align}
		\int_{\R^n} \int_{\R^n}  (f(x)-f(y))_+^p \mathbf{1}_{t K}(x-y) \, dy \, dx &\geq p \int_0^\infty \int_{\R^n} \int_{\R^n} (f^\star(x)-r)_+^{p-1} (1- \mathbf{1}_{\{f^\star \geq r\}}(y))  \mathbf{1}_{t K^\star}(x-y) \, dy \, dx \, dr \nonumber \\
		& = \int_{\R^n} \int_{\R^n}  (f^\star(x)-f^\star(y))_+^p \mathbf{1}_{t K^\star}(x-y) \, dy \, dx, \label{328}
	\end{align}
	where the last step follows from \eqref{326new} (after replacing $f$ and $K$ by $f^\star$ and $K^\star$, respectively).

	On the other hand, since
	$$
		(f(x)-f(y))_-^p = p \int_0^\infty (r-f(y))_-^{p-1} \mathbf{1}_{\{f < r \}}(x) \, dr
	$$
	one can apply the above reasoning line by line in order to get (cf. \eqref{328})
	\begin{equation}\label{329}
	\int_{\R^n} \int_{\R^n}  (f(x)-f(y))_-^p \mathbf{1}_{t K}(x-y) \, dy \, dx \geq  \int_{\R^n} \int_{\R^n}  (f^\star(x)-f^\star(y))_-^p \mathbf{1}_{t K^\star}(x-y) \, dy \, dx.
	\end{equation}
	
	Taking into account that $|f(x)-f(y)|^p = (f(x)-f(y))^p_+ + (f(x)-f(y))^p_-$ (and analogously for $f^\star$), we can invoke  \eqref{328}-\eqref{329} to achieve
	\begin{align*}
		 \int_{\R^n} \int_{\R^n}  |f(x)-f(y)|^p \mathbf{1}_{t K}(x-y) \, dy \, dx    \geq  \int_{\R^n} \int_{\R^n}  |f^\star(x)-f^\star(y)|^p \mathbf{1}_{t K^\star}(x-y) \, dy \, dx,
	\end{align*}
	or in other words (cf. \eqref{PS3} and its counterpart in terms of $f^\star$ and $K^\star$)
	$$
		 \int_{\|h\|_K < t} \|\Delta_h f\|^p_{L^p} \, d h  \geq  \int_{\|h\|_{K^\star} < t} \|\Delta_h f^\star\|^p_{L^p} \, d h.
	$$
	This completes the proof.
\end{proof}

\subsection{Anisotropic Besov norms}\label{SectionBesov}
As can be seen in  \cite{BerghLofstrom, DeVoreLorentz, Stein, Triebel}, Besov spaces are cornerstones in the theory of function spaces. There are several available approaches to introduce these spaces. Among them, the method relying on moduli of smoothness is of special relevance (cf. Remark \ref{RemarkAni} below). We start this subsection by presenting  anisotropic versions of Besov spaces, which are naturally introduced using anisotropic moduli of smoothness.

\begin{defn}\label{DefABs}
	Let $K$ be a star body in $\R^n$ and let $s \in (0, 1)$ and $p, q \in [1, \infty]$. The \emph{anisotropic Besov space} $B^s_{p, q; K}$ is formed by all $f \in L^p$ such that
\begin{equation*}
		\|f\|_{B^s_{p, q; K}} = \bigg(\int_0^\infty t^{-s q} \omega_K(f, t)_p^q \, \frac{dt}{t} \bigg)^{\frac{1}{q}} < \infty
	\end{equation*}
	(where the usual modification is made if $q=\infty$).
\end{defn}

\begin{rem}\label{RemarkAni}
	If $K = B^n$ then $B^s_{p, q; K} = B^s_{p, q}$, the classical Besov space given by
	\begin{equation}\label{ClassicB}
		\|f\|_{B^s_{p, q}} = \bigg(\int_0^\infty t^{-s q} \omega(f, t)_p^q \, \frac{dt}{t} \bigg)^{\frac{1}{q}},
	\end{equation}
	where $\omega(f, t)_p$ is defined by \eqref{CMod}.
\end{rem}

Now we show that anisotropic Besov spaces extend anisotropic fractional Sobolev spaces of Ludwig \cite{Ludwigb, Ludwig}:  the space $W^{s, p}_K, \, s \in (0, 1), \, p \in [1, \infty)$, is equipped with
$$
	\|f\|_{W^{s, p}_K} = \bigg( \int_{\R^n} \int_{\R^n} \frac{|f(x)-f(y)|^p}{\|x-y\|_K^{s p + n}} \, dx \, dy \bigg)^{\frac{1}{p}}.
$$
The case $K=B^n$ corresponds to the classical Gagliardo (semi-)norm \eqref{ClassicGag}.

\begin{prop}\label{Prop3.34}
	We have
	$$
		B^s_{p, p; K} = W^{s, p}_K	
	$$
	with
	$$
		\|f\|_{B^s_{p, p; K}}  = (|K|(s p + n))^{-\frac{1}{p}} \,  \|f\|_{W^{s, p}_K}.
	$$
\end{prop}

\begin{proof}
Changing the order of integration, we derive
\begin{align*}
		\|f\|_{B^s_{p, p; K}}^p &= \int_0^\infty t^{-s p} \omega_K(f, t)_p^p \, \frac{dt}{t}  = \frac{1}{|K|}  \int_0^\infty t^{-s p-n}  \int_{\|h\|_K < t} \|\Delta_h f\|^p_{L^p} \, d h \, \frac{dt}{t} \\
		& =  \frac{1}{|K|}   \int_{\R^n} \|\Delta_h f\|^p_{L^p} \int_{\|h\|_K}^\infty t^{-s p - n} \, \frac{dt}{t} \, dh  = \frac{1}{|K| (s p + n)} \int_{\R^n}  \frac{\|\Delta_h f\|^p_{L^p} }{\|h\|_K^{s p + n}} \, dh \\
		& =  \frac{1}{|K|(s p + n)} \,  \|f\|^p_{W^{s, p}_K}.
	\end{align*}
\end{proof}

\begin{thm}\label{ThmPSv1}
Let $s \in (0, 1), p \in [1, \infty),$ and $q \in [1, \infty]$. If $f \in B^s_{p, q; K}$ then
	\begin{equation}\label{ThmPS1}
		\|f^\star\|_{B^s_{p, q; K^\star}} \leq \|f\|_{B^s_{p, q; K}}.
	\end{equation}
	In particular, letting $p=q$ (cf. \cite[Theorem 11]{HaddadLudwig})
	$$
		\|f^\star\|_{W^{s, p}_{ K^\star}} \leq \|f\|_{W^{s, p}_{ K}}.
	$$
\end{thm}
\begin{proof}
	It is an immediate consequence of Proposition \ref{ThmPS}.
\end{proof}

Anisotropic Besov spaces can be characterized in terms of differences as follows.

\begin{prop}\label{PropositionDiff}
Let $s \in (0, 1)$ and $p, q \in [1, \infty)$. We have
	$$
		\|f\|_{B^s_{p, q; K}} \approx \bigg(\frac{1}{|K|} \int_{\R^n} \frac{\|\Delta_h f\|_{L^p}^q}{\|h\|_K^{s q + n}} \, dh \bigg)^{\frac{1}{q}},
	$$
	where the hidden equivalence constants depend only on $n$.
\end{prop}

\begin{proof}
	By virtue of Proposition \ref{LemmaTechnical},
	$$
		\omega_K(f, t)_p \approx \bigg(\frac{1}{t^n |K|} \int_{\|h\|_K < t} \|\Delta_h f\|^q_{L^p} \, d h \bigg)^{\frac{1}{q}},
	$$
	where equivalence constants depend only on $n$. Then,  a change of the order of integration yields
	\begin{equation*}
		\|f\|_{B^s_{p, q; K}}^q \approx \frac{1}{|K|} \int_0^\infty t^{-s q-n} \int_{\|h\|_K < t} \|\Delta_h f\|^q_{L^p} \, d h \, \frac{dt}{t} \approx \frac{1}{|K|} \int_{\R^n} \|h\|_K^{-s q -n} \|\Delta_h f\|^q_{L^p}  \, dh.
	\end{equation*}
\end{proof}

\subsection{Besov polar projection bodies: definition and basic properties}\label{SectionBPPB} We extend the notion of  the fractional polar projection body $\Pi^{*, s}_p$ due to Haddad and Ludwig \cite{HaddadLudwig, HaddadLudwig2} to the Besov polar projection body $\Pi^{*, s}_{p, q}$. Such an extension is made in a natural way, in the same fashion as fractional Sobolev spaces $W^{s, p}$ are part of the more general class formed by Besov spaces $B^s_{p, q}$; recall that $W^{s, p} = B^s_{p, p}$ (cf. Proposition \ref{Prop3.34}).

To avoid repetition, throughout this subsection, we always assume that $s \in (0, 1)$ and $p, q \in [1, \infty)$.

\begin{defn}\label{DefBPPB}
	 For a measurable function $f: \R^n \to \R$, \emph{Besov polar projection bodies} $\Pi^{*, s}_{p, q} f$ are defined as the star-shaped sets in $\R^n$ with gauge function given by
	$$
		\|\xi\|^{s q}_{\Pi^{*, s}_{p, q} f} = \int_0^\infty t^{-s q} \|\Delta_{t \xi} f\|_{L^p}^q \, \frac{dt}{t}, \qquad \xi \in \R^n.
	$$
\end{defn}

\begin{rem}\label{Rem3.38}
	It is obvious that $\Pi^{*, s}_{p, p} f = \Pi^{*, s}_p f$, cf. \eqref{FPPB}.
\end{rem}

\begin{rem}
		The set $\Pi^{*, s}_{p, q}f$ is invariant under translations of $f$ and volume preserving affine transformations:
	\begin{equation}\label{ProofOptimAff0}
		\Pi^{*, s}_{p, q} (f \circ \phi) = \phi^{-1} \Pi^{*, s}_{p, q} f
	\end{equation}
	for every  $\phi \in \text{SL}(n)$. Indeed, for every $\xi \in \R^n$,
	\begin{equation*}
		\|\xi\|^{s q}_{\Pi^{*, s}_{p, q} (f \circ \phi)} = \int_0^\infty t^{-s q} \|\Delta_{t \xi} (f \circ \phi)\|_{L^p}^q \, \frac{dt}{t}  = \int_0^\infty t^{-s q} \|\Delta_{t \phi \xi} f\|_{L^p}^q \, \frac{dt}{t}  = \|\phi \xi\|_{\Pi^{*, s}_{p, q}}^{s q}.
	\end{equation*}
\end{rem}

The relation between the measure of Besov polar projection bodies and classical Besov norms is given in the following result.

\begin{prop}\label{PropHolder}
The following inequality holds
	$$
		n \omega_n^{\frac{n + s q}{n}}  |\Pi^{*, s}_{p, q} f|^{-\frac{s q}{n}} \leq  \int_{\R^n} |h|^{-s q} \|\Delta_h f\|^q_{L^p} \, \frac{dh}{|h|^n}.
	$$
	In particular
	\begin{equation}\label{RelationB}
		 |\Pi^{*, s}_{p, q} f|^{-\frac{s q}{n}}  \lesssim \|f\|^q_{B^s_{p, q}}.
	\end{equation}
\end{prop}
\begin{proof}
	Applying H\"older's inequality with exponent $\rho = 1 + \frac{s q}{n}$, we have
	\begin{align*}
		|\S^{n-1}|^\rho &\leq \bigg( \int_{\S^{n-1}} \|\xi\|^{s q}_{\Pi^{*, s}_{p, q} f}  \, d \xi \bigg) \bigg( \int_{\S^{n-1}}  \|\xi\|^{-\frac{s q \rho'}{\rho}}_{\Pi^{*, s}_{p, q} f}  \, d \xi \bigg)^{\frac{\rho}{\rho'}} \\
		& = \bigg(\int_{\S^{n-1}}  \int_0^\infty t^{-s q} \|\Delta_{t \xi} f\|_{L^p}^q \, \frac{dt}{t} \, d \xi \bigg)  \bigg( \int_{\S^{n-1}}  \|\xi\|^{-n}_{\Pi^{*, s}_{p, q} f}  \, d \xi \bigg)^{\frac{s q}{n}} \\
		& = \bigg(\int_{\R^n} |h|^{-s q} \|\Delta_h f\|^q_{L^p} \, \frac{dh}{|h|^n} \bigg) (n |\Pi^{*, s}_{p, q} f|)^{\frac{s q}{n}}.
	\end{align*}
	Note that $|\S^{n-1}|= n \omega_n$. Then we have shown that
	$$
		n \omega_n^{\frac{n+ s q}{n}} \leq  \bigg(\int_{\R^n} |h|^{-s q} \|\Delta_h f\|^q_{L^p} \, \frac{dh}{|h|^n} \bigg)  |\Pi^{*, s}_{p, q} f|^{\frac{s q}{n}}.
	$$
\end{proof}

\begin{prop}\label{Prop348}
We have
\begin{equation}\label{347}
	\|f\|_{B^s_{p, q; \Pi^{*, s}_{p, q} f^\star}} = \omega_n^{-\frac{s}{n}} |\Pi^{*, s}_{p, q} f^\star|^{\frac{s}{n}} \|f\|_{B^s_{p, q}},
\end{equation}
and
\begin{equation}\label{347*}
		\frac{1}{4 (2^n+1)^{\frac{1}{p}}} \leq \bigg( \frac{n}{s q + n} \bigg)^{-\frac{1}{q}} \|f\|_{B^s_{p, q; \Pi^{*, s}_{p, q} f}} \leq 4 (2^n + 1)^{\frac{1}{q}}.
	\end{equation}
\end{prop}
\begin{proof}
We first note that $\Pi^{*, s}_{p, q} f^\star$ is a ball since $f^\star$ is radially symmetric (use that $\|\Delta_{t A \xi} f^\star\|_{L^p} = \|\Delta_{t \xi} f^\star\|_{L^p}$ if $A \in O(n)$). In particular $\|h\|_{\Pi^{*, s}_{p, q} f^\star} = \omega_n^{1/n} |\Pi^{*, s}_{p, q} f^\star|^{-1/n} |h|$, that yields to the following relation between anisotropic and classical moduli of smoothness  (cf. \eqref{CMod})
\begin{align*}
	\omega_{ \Pi^{*, s}_{p, q} f^\star} (f, t)_p & = \bigg(\frac{1}{| \Pi^{*, s}_{p, q} f^\star| t^n} \int_{ |h| < t |\Pi^{*, s}_{p, q} f^\star|^{1/n} \omega_n^{-1/n} } \|\Delta_h f\|^p_{L^p(\R^n)} \, d h \bigg)^{\frac{1}{p}}  = \omega (f,  t |\Pi^{*, s}_{p, q} f^\star|^{1/n} \omega_n^{-1/n} )_p.
\end{align*}
Inserting this into the definition of anisotropic Besov space given in Definition \ref{DefABs} and making a change of variables, we obtain
\begin{align*}
	\|f\|_{B^s_{p, q; \Pi^{*, s}_{p, q} f^\star}} &=  \bigg(\int_0^\infty t^{-s q} \omega_{\Pi^{*, s}_{p, q} f^\star}(f, t)_p^q \, \frac{dt}{t} \bigg)^{\frac{1}{q}} \\
	& =   \bigg(\int_0^\infty t^{-s q} \omega(f, t |\Pi^{*, s}_{p, q} f^\star|^{\frac{1}{n}} \omega_n^{-\frac{1}{n}})_p^q \, \frac{dt}{t} \bigg)^{\frac{1}{q}} \\
	& = \omega_n^{-\frac{s}{n}}  |\Pi^{*, s}_{p, q} f^\star|^{\frac{s}{n}}  \bigg(\int_0^\infty t^{-s q} \omega(f, t)_p^q \, \frac{dt}{t} \bigg)^{\frac{1}{q}}.
\end{align*}
This proves \eqref{347}.

Let us now show \eqref{347*}. We concentrate on the upper estimate, since the lower estimate can be obtained using similar ideas. It follows from \eqref{Lem1} and Fubini's theorem that
\begin{align*}
	\|f\|_{B^s_{p, q; K}} & = \bigg(\int_0^\infty t^{-s q} \omega_K(f, t)_p^q \, \frac{dt}{t} \bigg)^{\frac{1}{q}} \\
	& \leq 4 (2^n + 1)^{1/q} |K|^{-\frac{1}{q}} \bigg(\int_0^\infty t^{-s q-n}   \int_{\|h\|_K < t} \|\Delta_h f\|^q_{L^p(\R^n)} \, d h \, \frac{dt}{t} \bigg)^{\frac{1}{q}} \\
	& = 4 (2^n + 1)^{1/q} |K|^{-\frac{1}{q}}  \bigg(\int_{\R^n} \|\Delta_h f\|^q_{L^p}  \int_{\|h\|_K}^\infty t^{-s q-n} \, \frac{dt}{t} \, dh  \bigg)^{\frac{1}{q}} \\
	& =4 (2^n + 1)^{1/q} |K|^{-\frac{1}{q}} (s q + n)^{-\frac{1}{q}}  \bigg(\int_{\R^n} \|h\|_K^{-s q -n} \|\Delta_h f\|^q_{L^p} \, dh\bigg)^{\frac{1}{q}}.
\end{align*}
Then, using polar coordinates,
\begin{align*}
	\|f\|_{B^s_{p, q; K}}  &\leq 4 (2^n + 1)^{1/q}  |K|^{-\frac{1}{q}} (s q + n)^{-\frac{1}{q}}   \bigg(\int_{\S^{n-1}} \|\xi\|_K^{-s q -n} \int_0^\infty  t^{-s q } \|\Delta_{t \xi} f\|^q_{L^p}  \, \frac{dt}{t} \, d \xi \bigg)^{\frac{1}{q}} \\
	& = 4 (2^n + 1)^{1/q} |K|^{-\frac{1}{q}} (s q + n)^{-\frac{1}{q}}   \bigg(\int_{\S^{n-1}} \|\xi\|_K^{-s q -n} \|\xi\|_{\Pi^{*, s}_{p, q} f}^{s q}   \, d \xi \bigg)^{\frac{1}{q}}.
\end{align*}
In particular, letting $K = \Pi^{*, s}_{p, q} f$ and using \eqref{Volume}, the previous estimate reads as
\begin{align*}
	\|f\|_{B^s_{p, q;  \Pi^{*, s}_{p, q} f}}  &\leq 4 (2^n + 1)^{1/q}  |\Pi^{*, s}_{p, q} f|^{-\frac{1}{q}} (s q + n)^{-\frac{1}{q}}   \bigg(\int_{\S^{n-1}} \|\xi\|_{\Pi^{*, s}_{p, q} f}^{-n}   \, d \xi \bigg)^{\frac{1}{q}} \\
	& =4 (2^n + 1)^{1/q}  |\Pi^{*, s}_{p, q} f|^{-\frac{1}{q}} (s q + n)^{-\frac{1}{q}}  (n |\Pi^{*, s}_{p, q} f|)^{\frac{1}{q}}  \\
	&= 4 \bigg( \frac{n (2^n +1)}{s q + n} \bigg)^{\frac{1}{q}}.
\end{align*}
This proves the upper estimate in \eqref{347*}.
\end{proof}

\begin{rem}
	Constants in \eqref{347*} are not sharp. For instance, one can carry out the above argument but now relying on \eqref{Lem2} (rather than \eqref{Lem1}) and then
	$$
		 \bigg( \frac{n}{s q + n} \bigg)^{-\frac{1}{q}} \|f\|_{B^s_{p, q; \Pi^{*, s}_{p, q} f}} \leq 1 \qquad \text{if} \qquad p \leq q
	$$
	and
	$$
		 \bigg( \frac{n}{s q + n} \bigg)^{-\frac{1}{q}} \|f\|_{B^s_{p, q; \Pi^{*, s}_{p, q} f}} \geq 1  \qquad \text{if} \qquad p \geq q.
	$$
	In the special case $p=q$, from Proposition \ref{Prop3.34} and Remark \ref{Rem3.38}, we derive the following formula
	$$
		 n^{-\frac{1}{p}} |\Pi^{*, s}_{p} f|^{-\frac{1}{p}} \|f\|_{W^{s, p}_{\Pi^{*, s}_{p} f}} = 1.
	$$
\end{rem}

\subsection{The P\'olya--Szeg\H{o} inequality for Besov polar projection bodies}\label{SectionPSB}
The classical P\'olya--Szeg\H{o} inequality for Besov norms asserts that (cf. \eqref{PointEstimRea} and \eqref{ClassicB})
\begin{equation}\label{ClassicPSB}
	\|f^\star\|_{B^s_{p, q}} \leq \|f\|_{B^s_{p, q}}.
\end{equation}
Next we establish the affine version of this inequality.

\begin{thm}\label{ThmPSBod}
We have
	\begin{equation}\label{PSbod}
		|\Pi^{*, s}_{p, q} f^\star|^{-\frac{s q}{n}} \leq 4^q (2^n + 1)^{\max\{\frac{q}{p}, 1\}}\,  |\Pi^{*, s}_{p, q} f|^{-\frac{s q}{n}}.
	\end{equation}
	If $p=q$ the above constant $4^p (2^n + 1)$ can be replaced by $1$.
\end{thm}

\begin{rem}
	Inequality \eqref{PSbod} is stronger than \eqref{ClassicPSB} in the sense that
	$$
		\|f^\star\|_{B^s_{p, q}}^q  \approx |\Pi^{*, s}_{p, q} f^\star|^{-\frac{s q}{n}} \lesssim |\Pi^{*, s}_{p, q} f|^{-\frac{s q}{n}} \lesssim \|f\|_{B^s_{p, q}}^q.
	$$
	The first estimate easily follows from \eqref{347} and \eqref{347*} applied to $f^\star$. Concerning the last estimate, see \eqref{RelationB}.
\end{rem}

As in  \cite[Theorem 13]{HaddadLudwig}, where the case $p=q$ was already considered, the proof of \eqref{PSbod} is a simple consequence of  the P\'olya--Szeg\H{o} inequalities stated in Theorem \ref{ThmPSv1} via the dual mixed volume inequality. However, in sharp contrast with the case $p=q$, the general case $p \neq q$ requires additional technicalities that essentially come from the fact that
\begin{equation}\label{Diffbn}
	\|f\|_{B^s_{p, q; K}} = \bigg(\frac{1}{|K|} \int_{\R^n} \|h\|_K^{-s q - n} \|\Delta_h f\|_{L^p}^q \, dh \bigg)^{1/q}
\end{equation}
is no longer true, but we only have the weaker assertion obtained when $=$ in \eqref{Diffbn} is replaced by $\approx$ (cf. Proposition \ref{PropositionDiff}). This is the reason behind the additional constant $4^q (2^n + 1)^{\max\{\frac{q}{p}, 1\}}$ in \eqref{PSbod}, which does not appear if $p=q$. It may be well the case that this constant is not optimal\footnote{Optimal constants in P\'olya--Szeg\H{o} inequalities for moduli of smoothness are a rather delicate issue, even in the classical setting $K = B^n$.}. However, in later applications, we will be mainly interested in the special choice of parameters $q=n < p$ and then  $4^n (2^n + 1) \approx 8^n$.

\begin{proof}[Proof of Theorem \ref{ThmPSBod}]
Let $K \subset \R^n$ be a star-shaped set. By polar coordinates, we have
\begin{align}
	\tilde{V}_{-s q} (K, \Pi^{*, s}_{p, q} f) & = \frac{1}{n} \int_{\S^{n-1}} \|\xi\|_{K}^{-n-s q} \|\xi\|^{s q}_{\Pi^{*, s}_{p, q} f} \, d \xi \nonumber \\
	& =  \frac{1}{n} \int_{\S^{n-1}} \|\xi\|_{K}^{-n-s q}  \int_0^\infty t^{-s q} \|\Delta_{t \xi} f\|_{L^p}^q \, \frac{dt}{t} \, d \xi \nonumber \\
	& =  \frac{1}{n}  \int_{\R^n} \|h\|_{K}^{-n-s q} \|\Delta_h f\|_{L^p}^q \, d h.\label{ThmPSBod1}
\end{align}
	
	Since we are interested in applying \eqref{ThmPS1}, we need to replace the Besov norm given in \eqref{ThmPSBod1} by the equivalent one  $\|f\|_{B^s_{p, q;K}}$. This is possible through the application of Lemma \ref{LemmaTechnical}. Assume first that $q < p$. It follows from Fubini's theorem and \eqref{Lem1} that
	\begin{align}
		 \int_{\R^n} \|h\|_{K}^{-n-s q} \|\Delta_h f\|_{L^p}^q \, d h & = (n+s q) \int_{\R^n}  \|\Delta_h f\|_{L^p}^q \int_{\|h\|_K}^\infty t^{-n - s q} \, \frac{dt}{t} \, dh \nonumber \\
		 & = (n+s q) |K|  \int_0^\infty t^{- s q} \frac{1}{t^n |K|} \int_{\|h\|_K < t}  \|\Delta_h f\|_{L^p}^q \, dh \, \frac{dt}{t} \nonumber\\
		 & \geq  \frac{n+s q}{4^q (2^n + 1)} |K| \int_0^\infty t^{-s q} \bigg(\frac{1}{t^n |K|} \int_{\|h\|_K < t} \|\Delta_h f\|_{L^p}^p \, dh \bigg)^{q/p} \, \frac{dt}{t} \nonumber \\
		 & = \frac{n+s q}{4^q (2^n + 1)} |K| \, \|f\|^q_{B^s_{p, q; K}}.\label{ThmPSBod2}
	\end{align}
	Conversely, by \eqref{Lem2}, one can easily prove that
	\begin{equation}\label{ThmPSBod3}
	 \int_{\R^n} \|h\|_{K}^{-n-s q} \|\Delta_h f\|_{L^p}^q \, d h \leq (n + s q) |K| \, \|f\|^q_{B^s_{p, q; K}}.
	\end{equation}

	It follows now from  \eqref{ThmPSBod1}, \eqref{ThmPSBod2}, \eqref{ThmPS1} and \eqref{ThmPSBod3} (applied to $f^\star$ and $K^\star$) that
	\begin{align*}
		\tilde{V}_{-s q} (K, \Pi^{*, s}_{p, q} f) & \geq \frac{n+s q}{4^q (2^n + 1) n} |K| \, \|f^\star\|^q_{B^s_{p, q; K^\star}} \\
		 & \geq \frac{1}{4^q (2^n + 1) n}  \int_{\R^n} \|h\|_{K^\star}^{-n-s q} \|\Delta_h f^\star\|_{L^p}^q \, d h,
	\end{align*}
	and invoking again  \eqref{ThmPSBod1}, but now applied to $f^\star, K^\star$, we achieve
	\begin{equation}\label{ThmPSBod4}
		\tilde{V}_{-s q} (K, \Pi^{*, s}_{p, q} f) \geq  \frac{1}{4^q (2^n + 1)}   	\tilde{V}_{-s q} (K^\star, \Pi^{*, s}_{p, q} f^\star).
	\end{equation}
	
	The analog of \eqref{ThmPSBod4} under $q \geq p$ can be analogously obtained and it reads
	\begin{equation}\label{ThmPSBod4*}
		\tilde{V}_{-s q} (K, \Pi^{*, s}_{p, q} f) \geq  \frac{1}{4^q (2^n + 1)^{\frac{q}{p}}}   	\tilde{V}_{-s q} (K^\star, \Pi^{*, s}_{p, q} f^\star).
	\end{equation}

	At this point, the desired estimate \eqref{PSbod} can be obtained by mimicking  the proof of  \cite[Theorem 13]{HaddadLudwig}  related to the case $p=q$. For convenience of the reader, we provide below with full details.
	
	Applying the dual mixed volume inequality (which is a consequence of H\"older's inequality)
	$$
			\tilde{V}_{-s q} (K^\star, \Pi^{*, s}_{p, q} f^\star) \geq |K|^{1 + \frac{s q}{n}} |\Pi^{*, s}_{p, q} f^\star|^{-\frac{s q}{n}}.	
	$$
	Hence, by \eqref{ThmPSBod4} and \eqref{ThmPSBod4*},
	$$
		\tilde{V}_{-s q} (K, \Pi^{*, s}_{p, q} f) \geq  \frac{1}{4^q (2^n + 1)^{\max\{\frac{q}{p}, 1\}}} \, |K|^{1 + \frac{s q}{n}} |\Pi^{*, s}_{p, q} f^\star|^{-\frac{s q}{n}}.
	$$
	In particular, taking $K = \Pi^{*, s}_{p, q} f$ in the previous estimate,
	$$
		 |\Pi^{*, s}_{p, q} f|^{-\frac{s q}{n}} \geq  \frac{1}{4^q (2^n + 1)^{\max\{\frac{q}{p}, 1\}}} \, |\Pi^{*, s}_{p, q} f^\star|^{-\frac{s q}{n}}.
	$$
\end{proof}

\subsection{Connections between $\Pi^{*, s}_{p, q} f$ and $\Pi^*_p f$}\label{SectionCBC}
Applying similar techniques as in  \cite[Theorem 10]{HaddadLudwig} (which corresponds to the special case $p=q$; see also \eqref{HLLimit}), one can show that
\begin{equation}\label{HLLimit2}
		\lim_{s \to 1-} (q (1-s))^{\frac{1}{q}} \|\xi\|_{\Pi^{*, s}_{p, q} f} =  \|\xi\|_{\Pi^*_p f}
\end{equation}
for every $f \in W^1_p$ and $\xi \in \S^{n-1}$.
The formula \eqref{HLLimit2} may be viewed as directional counterparts of known extensions (cf. \cite{KMX, KL}) of the Bourgain-Brezis-Mironescu formula \eqref{Intro3} from $W^{s, p} = B^s_{p, p}$ to the full scale of Besov spaces $B^s_{p, q}$.

In our further considerations,  we shall need the following variant  of \eqref{HLLimit2} where the limit is not only taken with respect to the smoothness parameter $s$, but also with respect to the integrability parameter $p$. Related results in the classical setting of Gagliardo-type functionals and total variation may be found in \cite{BN16}.

\begin{theorem}\label{Th1}
Let $\xi\in\mathbb{S}^{n-1}$.
Then, for any $f\in C_c^2(\R^n)$,
\begin{equation}\label{Thm1for1}
\lim\limits_{p\to n+}
\left(1-\frac np\right)^{\frac{p}{n}}  \|\xi\|_{\Pi^{*, \frac{n}{p}}_{p, n} f}^{n}  =
\frac1n \,  \|\xi\|_{\Pi^*_n f}^n.
\end{equation}
\end{theorem}

Formula \eqref{Thm1for1} is a simple application of the following result on convergence of functionals involving families of nonnegative functions $\{\rho_{\varepsilon}\}_{\varepsilon\in(0,\infty)}$ on $(0,\infty)$
satisfying that,
for any $\varepsilon\in(0,\infty)$,
\begin{align}\label{e1}
\int_{0}^\infty\rho_{\varepsilon}(t)\,dt=1
\end{align}
and, for any $a\in(0,1)$,
\begin{align}\label{e2}
\lim\limits_{\varepsilon\to0+}
\int_{0}^{a}\rho_{\varepsilon}(t)\,dt=1.
\end{align}

\begin{lem}\label{Th2}
Let $p\in(0,\infty)$ and $\{\rho_{\varepsilon}\}_{\varepsilon\in(0,\infty)}$
be a family of nonnegative functions on $(0,\infty)$
satisfying both \eqref{e1} and \eqref{e2}.
Assume $g$ is a nonnegative function on
$(0,\infty)\times(0,\infty)$ satisfying that\footnote{$\lim_{t, \varepsilon \to 0+} f(t, \varepsilon) = L$ means that given any $\eta > 0$ there exists $\delta > 0$ such that if $t, \varepsilon \in (0, \delta)$ then $|f(t, \varepsilon)-L| < \eta$.}
\begin{align*}
L:=\lim\limits_{t,\varepsilon\to0+}\frac{g(t,\varepsilon)}{t}
\in[0,\infty)\qquad \text{and}\qquad
\limsup_{\varepsilon\to0+}\sup_{t\in(0,\infty)}
\frac{g(t,\varepsilon)}{t}<\infty.
\end{align*}
Then
\begin{align*}
\lim_{\varepsilon\to0+}
\left\{\int_{0}^{\infty}\left[\frac{g(t,\varepsilon)}{t}\right]^p
\rho_{\varepsilon}(t)\,dt\right\}^{\frac1p}=L.
\end{align*}
\end{lem}

The proof of this result follows similar ideas as in \cite[Lemma 2.8(i)]{DLTYY}, where   the special case $g(t, \varepsilon) = g(t), \, \varepsilon > 0,$ was already considered.

\begin{proof}[Proof of Theorem \ref{Th1}]
Observe that
$$
	\int_1^\infty \left[t^{-\frac{n}{p}} \|\Delta_{t \xi} f\|_{L^p} \right]^n \, \frac{dt}{t} \leq  \frac{2^n p }{n^2} \, \|f\|_{L^p}^n.
$$
Therefore, to show \eqref{Thm1for1},
we only need to check that
\begin{align}\label{Th1-e1}
\lim\limits_{p\to n+}n
\left(1-\frac np\right)\int_{0}^{1}
\left[t^{-\frac np}\left\|\Delta_{t\xi}f\right\|_{L^p}
\right]^n\,\frac{dt}{t}=
\int_{\rn}\left|\nabla f(x)\cdot\xi\right|^n\,dx.
\end{align}
In order to show this, we prove
\begin{align}\label{Th1-e2}
\lim_{\varepsilon\to0+}\int_{0}^{\infty}\left[
\frac{\|\Delta_{t\xi}f\|_{L^{p(\varepsilon)}}}{t}\right]^n
\rho_{\varepsilon}(t)\,dt
=\int_{\rn}\left|\nabla f(x)\cdot\xi\right|^n\,dx,
\end{align}
where the sequence $\{\rho_{\varepsilon}\}_{\varepsilon\in(0,\infty)}$
of nonnegative functions on $(0,\infty)$ satisfies both
\eqref{e1} and \eqref{e2} and $p(\cdot):\ (0,\infty)\to[1,\infty)$
is a nonnegative function satisfying $\lim_{\varepsilon\to0+}
p(\varepsilon)=n$.
Indeed, if we assume \eqref{Th1-e2} holds for the moment,
then, letting $\rho_{\varepsilon}(t):=\varepsilon t^{\varepsilon-1}{\bf 1}_{(0,1)}(t)$,
$\varepsilon:=n(1-\frac np)$, and $p(\varepsilon):=\frac{n^2}{n-\varepsilon}$,
we obtain \eqref{Th1-e1} and hence complete the proof of \eqref{Thm1for1}.

Thus, it suffices to show \eqref{Th1-e2}.
Let $g(t,\varepsilon)=\|\Delta_{t\xi}f\|_{L^{p(\varepsilon)}}$.
By Lemma \ref{Th2}, we only require to prove
\begin{align}\label{Th1-e3}
\lim_{t,\varepsilon\to0+}\frac{g(t,\varepsilon)}{t}
=\left\|\nabla f\cdot \xi\right\|_{L^n}
\qquad  \text{and}\qquad
\limsup_{\varepsilon\to0+}
\sup_{t\in(0,\infty)}\frac{g(t,\varepsilon)}{t}<\infty.
\end{align}
Notice that $g(t,\varepsilon)
\leq  t \left\|\nabla f\right\|_{L^{p(\varepsilon)}}$ (cf. \eqref{314b2}) and then $\limsup_{\varepsilon\to0+}
\sup_{t\in(0,\infty)}\frac{g(t,\varepsilon)}{t} \leq \|\nabla f\|_{L^n} < \infty$. Without loss of generality, assume that $\text{supp } f \subset B(0, R)$ for some $R > 0$ and $t \in (0, R)$. Then, from triangle inequality and \eqref{LipEstim} (with $h = t \xi$),
\begin{align*}
\left|\frac{\|\Delta_{t\xi}f\|_{L^{p(\varepsilon)}}}{t}
-\left\|\nabla f\cdot\xi\right\|_{L^n}\right| &\le \left\| \frac{\Delta_{t \xi} f}{t} - \nabla f \cdot \xi \right\|_{L^{p(\varepsilon)}}  + \left|\left\|\nabla f\cdot\xi\right\|_{L^{p(\varepsilon)}}
-\left\|\nabla f\cdot\xi\right\|_{L^n}\right| \\
&\leq A  |B(0, 2R)|^{\frac{1}{p(\varepsilon)}} t
+\left|\left\|\nabla f\cdot\xi\right\|_{L^{p(\varepsilon)}}
-\left\|\nabla f\cdot\xi\right\|_{L^n}\right|.
\end{align*}
Letting $t,\varepsilon\to0+$, we obtain \eqref{Th1-e3}.
\end{proof}

%%%%%%%%%%%%%%%%%%%
\section{Sharp Besov inequalities involving Euclidean norms}\label{Section4}

\subsection{Sharp Besov inequalities in the critical case}

Let $n \geq 2$. It is well known that $B^{n/p}_{p, n}$, the classical Besov space given by \eqref{ClassicB}, is not embedded into $L^\infty$ but it is (locally) embedded into $L^q, \, q < \infty$. That is, there exists a constant $c = c(n, p, q, |\text{supp } f|)$ such that
\begin{equation}\label{EmbBBM}
	\|f\|_{L^q} \leq c \|f\|_{B^{n/p}_{p, n}}
\end{equation}
for every $f \in B^{n/p}_{p, n}$ with $|\text{supp } f| < \infty$. However, the sharp value of the constant $c$ is unknown. The following result gives a sharp version of \eqref{EmbBBM} in terms of the behaviour of $c$ with respect to both integrability parameters $p$ and $q$.

\begin{thm}\label{ThmEmbBesov1}
Let\footnote{The unessential constant $p_0$ indicates that only values $p \to n+$ and $q \to \infty$ are of some interest.} $2 \leq n < p \leq p_0 \leq  q < \infty$. Then there exists a constant $c_n$, depending only on $n$, such that
	\begin{equation}\label{PropBBM}
		\|f\|_{L^q} \leq c_n |\emph{supp } f|^{1/q}    q^{1/n'} \Big(1-\frac{n}{p} \Big)^{1/n} \|f\|_{B^{n/p}_{p, n}}
	\end{equation}
	for every $f \in B^{n/p}_{p, n}$ with $|\emph{supp } f| < \infty$.
\end{thm}

Inequality \eqref{PropBBM} reveals that  $\|f\|_{L^q}$ increases as $q^{1/n'}$ and decreases as $(1-\frac{n}{p})^{1/n}$ provided that $q \to \infty$ and $p \to n+$, respectively. The proof relies on a variety of techniques, including wavelets and interpolation theory.

\begin{proof}[Proof of Theorem \ref{ThmEmbBesov1}]
Let $\{\psi_{j m} : j \in \Z, m \in \Z^n\} = \{\psi_{j m}\}$ be the system generated by the (\emph{finite}) wavelet set $\Psi$, specifically, $\psi_{j m}(x) = 2^{j n/2} \psi(2^j x - m), \, \psi \in \Psi$. Without loss of generality, we may assume that every $\psi$ is supported on $Q_0$, the unit cube, and $\{\psi_{j m}\}$ forms an orthonormal basis in $L^2$. In particular, the function $f$ can be decomposed (at leat, in the distributional sense) as
\begin{equation}\label{PBBM0}
	f = \sum_{\psi \in \Psi} \sum_{j=-\infty}^\infty \sum_{m \in \Z^n} \lambda_{j m}(f) \psi_{j m},
\end{equation}
where
$$\lambda_{j m}(f) =  \int_{\R^n} f(x) \psi_{j m}(x) \, dx,$$
the wavelet coefficients of $f$.

Next we estimate the $L^q$-norm of  $f$ given by \eqref{PBBM0}.  Indeed, using the triangle inequality and the fact that, for every fixed $j$, the wavelets $\{\psi_{j m}: m \in \Z^n\}$ have pairwise disjoint supports, we obtain\footnote{The index set $\psi \in \Psi$ does not play any role in our arguments and can be safely omitted.}
\begin{equation}
	\|f\|_{L^q}  \leq \sum_{j=-\infty}^\infty \bigg\|\sum_{m \in \Z^n} \lambda_{j m}(f) \psi_{j m} \bigg\|_{L^q} = \sum_{j=-\infty}^\infty  \bigg(\sum_{m \in \Z^n} |\lambda_{j m}(f)|^q \|\psi_{j m}\|_{L^q}^q \bigg)^{1/q}. \label{PBBM0**}
\end{equation}
Further, a simple change of variables together with the fact that $\|\psi\|_{L^q} \approx 1$ (with underlying absolute constants) lead to
\begin{align*}
	\|\psi_{j m}\|_{L^q} & = 2^{j n/2} \bigg(\int_{\R^n} |\psi(2^j x - m)|^q \, dx \bigg)^{1/q}  \\
	&= 2^{j n/2} 2^{-j n/q} \|\psi\|_{L^q}  \lesssim 2^{j n/2} 2^{-j n/q} \\
	& \lesssim 2^{j n/2} 2^{-j n/q} \|\psi\|_{L^p}  \\
	& = 2^{j n/2} 2^{-j n/q} 2^{j n/p}  \bigg(\int_{\R^n} |\psi(2^j x - m)|^p \, dx \bigg)^{1/p}  \\
	& = 2^{j n (1/p - 1/q)} \|\psi_{j m}\|_{L^p}.
\end{align*}
Inserting this estimate into \eqref{PBBM0**} and using that $p \leq q$, we get
\begin{align*}
	\|f\|_{L^q} &\lesssim   \sum_{j=-\infty}^\infty 2^{j n(\frac{1}{p}-\frac{1}{q})}  \bigg(\sum_{m \in \Z^n} |\lambda_{j m}(f)|^q  \|\psi_{j m}\|_{L^p}^q \bigg)^{1/q} \\
	& \leq  \sum_{j=-\infty}^\infty 2^{j n(\frac{1}{p}-\frac{1}{q})}  \bigg(\sum_{m \in \Z^n} |\lambda_{j m}(f)|^p  \|\psi_{j m}\|_{L^p}^p \bigg)^{1/p} \\
	& \lesssim  \sum_{j=-\infty}^\infty 2^{j n(\frac{1}{2}-\frac{1}{q})} \bigg(\sum_{m \in \Z^n} |\lambda_{j m}(f)|^p   \bigg)^{1/p}.
\end{align*}
Choose $j_0 \in \Z$ such that $|\text{supp } f| \approx 2^{-j_0 n}$. The previous computations show that
\begin{align}
	\|f\|_{L^q} &\lesssim   \sum_{j=-\infty}^{j_0} 2^{j n(\frac{1}{2}-\frac{1}{q})} \bigg(\sum_{m \in \Z^n} |\lambda_{j m}(f)|^p   \bigg)^{1/p} +   \sum_{j=j_0}^\infty 2^{j n(\frac{1}{2}-\frac{1}{q})} \bigg(\sum_{m \in \Z^n} |\lambda_{j m}(f)|^p   \bigg)^{1/p} \nonumber \\
	& =: \mathcal{J}_1+ \mathcal{J}_2. \label{341n}
\end{align}

To estimate $\mathcal{J}_2$, we can apply  H\"older's inequality with exponent $n$ to obtain
\begin{align}
	\mathcal{J}_2  &\lesssim \left( \sum_{j=j_0}^\infty 2^{j n^2/2} \bigg(\sum_{m \in \Z^n} |\lambda_{j m}(f)|^p   \bigg)^{n/p} \right)^{1/n} \left(\sum_{j=j_0}^\infty 2^{-j n n'/q} \right)^{1/n'} \nonumber \\
	& \approx q^{1/n'} 2^{-j_0 n/q}  \left( \sum_{j=j_0}^\infty  \bigg(\sum_{m \in \Z^n} |2^{j n/2} \lambda_{j m}(f)|^p   \bigg)^{n/p} \right)^{1/n} \nonumber \\
	& \lesssim q^{1/n'} |\text{supp } f|^{1/q} \|f\|_{\mathfrak{B}^{n/p}_{p, n}}, \label{342n}
\end{align}
where
$$
	\|f\|_{\mathfrak{B}^s_{p, r}} := \left(\sum_{j=-\infty}^\infty 2^{j (s-n/p) r} \bigg(\sum_{m \in \Z^n} |2^{j n/2} \lambda_{j m}(f)|^p   \bigg)^{r/p} \right)^{1/r}.
$$
It is worthwhile to mention that, in sharp  contrast with the Besov norm $\|f\|_{B^s_{p, r}}$ given by moduli of smoothness (cf. \eqref{ClassicB}), the wavelet norm $\|f\|_{\mathfrak{B}^s_{p, r}}$  makes sense for any $s \in \R$.

Next we focus on $\mathcal{J}_1$. Given $j \leq j_0$, we observe that the cardinality of the set $\{m \in \Z^n : \lambda_{j m}(f) \neq 0\}$ depends only on $n$. Without loss of generality, we may assume that there exists $m_j \in \Z^n$ such that $\lambda_{j m}(f) =0$ for all $m \in \Z^n, m  \neq m_j$, and $\lambda_{j m_j}(f) = 2^{j n/2} \int_{\R^n} f$.  Therefore (recall that $q > 2$)
\begin{align}
	\mathcal{J}_1 &= \sum_{j=-\infty}^{j_0} 2^{j n (\frac{1}{2}-\frac{1}{q})}  |\lambda_{j m_j}(f)| \nonumber   \lesssim \bigg( \sum_{j=-\infty}^{j_0} 2^{j n (\frac{1}{2}-\frac{1}{q})} \bigg) \sup_{j \leq j_0}  |\lambda_{j m_j}(f)| \nonumber  \\
	& \lesssim \bigg(\frac{1}{2}-\frac{1}{q} \bigg)^{-1} 2^{j_0 n (\frac{1}{2}-\frac{1}{q})} \sup_{j \leq j_0}  2^{\frac{j n}{2}} \bigg|\int_{\R^n} f \bigg| \nonumber \\
	&  \approx 2^{j_0 n (1-\frac{1}{q})} \bigg|\int_{\R^n} f \bigg| \approx |\text{supp } f|^{\frac{1}{q}}  2^{\frac{j_0 n}{2}} |\lambda_{j_0 m_{j_0}} (f)| \nonumber  \\
	& \leq  |\text{supp } f|^{\frac{1}{q}} \|f\|_{\mathfrak{B}^{\frac{n}{p}}_{p, n}}. \label{343n}
\end{align}

As a consequence of \eqref{341n},  \eqref{342n} and \eqref{343n}, we achieve
\begin{equation}\label{344n}
	\|f\|_{L^q} \lesssim q^{1/n'} |\text{supp } f|^{1/q} \|f\|_{\mathfrak{B}^{n/p}_{p, n}}.
\end{equation}

%
%Consider the decomposition $f = \sum_{j=-\infty}^\infty L_j f$, where $L_j f$ denotes the standard dyadic frequencies of $f$. Since $\widehat{f}(\xi)= 0$ if $\xi \in B(0, 2^{j_0})$, then $L_j f =0$ for $j < j_0$.  Recall  the sharp version of the Nikolskii inequality given in \cite[Corollary 2]{NesselWilmes}:
%\begin{equation}\label{PBBM1}
%	\|L_j f\|_{L^q} \leq \bigg(\frac{2 \pi^{\frac{n}{2}}}{n \Gamma(\frac{n}{2})} \lceil{p}\rceil^n \bigg)^{\frac{1}{p}-\frac{1}{q}}   2^{j n (\frac{1}{p}-\frac{1}{q})} \|L_j f\|_{L^p}, \qquad j \in \Z,
%\end{equation}
%provided that $q > p$.
%Note that, under assumptions on $p$ and $q$, the constant $(\frac{2 \pi^{\frac{n}{2}}}{n \Gamma(\frac{n}{2})} \lceil{p}\rceil^n)^{\frac{1}{p}-\frac{1}{q}} \approx c_n$.
%It follows from \eqref{PBBM1} and H\"older's inequality that
%	\begin{align}
%	\|f\|_{L^q}  &\leq \sum_{j=j_0}^\infty \|L_j f\|_{L^q}  \nonumber \\
%	& \lesssim   \sum_{j=j_0}^\infty 2^{j n (\frac{1}{p}-\frac{1}{q})} \|L_j f\|_{L^p}  \nonumber\\
%	& \lesssim \bigg(\sum_{j=j_0}^\infty 2^{-j n n'/q} \bigg)^{1/n'}  \bigg(\sum_{j=j_0}^\infty (2^{j n/p} \|L_j f\|_{L^p})^n \bigg)^{1/n} 	\nonumber  \\
%	& \approx q^{1/n'} 2^{-j_0 n/q}   \bigg(\sum_{j=0}^\infty (2^{j n/p} \|L_j f\|_{L^p})^n \bigg)^{1/n}.  \label{PBBM2}
%	\end{align}
	
	It is well known (see e.g. \cite{Triebel08}) that $B^{n/p}_{p, n}$ given by \eqref{ClassicB} can be equivalently characterized in terms of wavelets and
	$$
		\|f\|_{B^{n/p}_{p, n}} \approx \|f\|_{\mathfrak{B}^{n/p}_{p, n}}.
	$$
	However, there are equivalence constants depending on $p$ and $n$. The behaviour of these constants are essential in our arguments. Then we will show that
	\begin{equation}\label{PBBM3}
		 \|f\|_{\mathfrak{B}^{n/p}_{p, n}} \lesssim  \Big(1-\frac{n}{p} \Big)^{1/n} \|f\|_{B^{n/p}_{p, n}}
	\end{equation}
	as $p \to n+$. Assuming momentarily the validity of such a claim,  the desired estimate \eqref{PropBBM} comes from \eqref{344n} and \eqref{PBBM3}.
	
	The proof of \eqref{PBBM3} makes use of some interpolation techniques. For the benefit of the non-expert reader in interpolation theory, we briefly recall the construction of the real interpolation method. Recall that the \emph{$K$-functional} relative to a (compatible\footnote{Loosely speaking, this means that $A_0 + A_1$ makes sense.}) Banach pair $(A_0, A_1)$ is defined by
	\begin{equation}\label{DefKFunct}
		K(t, f; A_0, A_1) =  \|f\|_{A_0 + t A_1}= \inf_{f = f_0 + f_1} (\|f_0\|_{A_0} + t \|f_1\|_{A_1})
	\end{equation}
	for every $t > 0$ and $f \in A_0+A_1$.  Let $\theta \in (0, 1)$ and $r \in [1, \infty)$. The \emph{real interpolation space} $(A_0, A_1)_{\theta, r}$ is the set of all those $f \in A_0 + A_1$ for which
	\begin{equation}\label{DefInt}
		\|f\|_{(A_0, A_1)_{\theta, r}} = \bigg[\int_0^\infty (t^{-\theta} K(t, f; A_0, A_1))^r \, \frac{dt}{t} \bigg]^{\frac{1}{r}} < \infty.
	\end{equation}
	For further details we refer the reader to e.g. \cite{BerghLofstrom, DeVoreLorentz}.

	The $K$-functional is an useful tool that can be characterized, for some classical Banach pairs, in terms of more familiar objects in analysis. For instance, it is well known that, for every given $p \in [1, \infty]$,
	\begin{equation}\label{EquivKfunctMod}
		K(t, f; L^p, \dot{W}^1_p) \approx \tilde{\omega}(f, t)_p:=  \sup_{|h| < t} \|\Delta_h f\|_{L^p},
	\end{equation}
	where  the hidden equivalence constants depend only on $n$ (see e.g. the explicit computations carried out in \cite{JohnenScherer}). Moreover, by \eqref{EquivMod} with $K = B^n$,
	$$
		  \tilde{\omega}(f, t)_p \leq 4 (2^n + 1)^{\frac{1}{p}} \omega(f, t)_p.
	$$
	As a byproduct, there exists a constant $c_n$, depending only on $n$, such that
	\begin{equation*}
		\|f\|_{(L^p, \dot{W}^1_p)_{n/p, n}} \leq c_n  \bigg(\int_0^\infty  [t^{-n/p} \omega(f, t)_p]^n \, \frac{dt}{t} \bigg)^{1/n}  = c_n \|f\|_{B^{n/p}_{p, n}}.
	\end{equation*}
	Hence the proof of \eqref{PBBM3} can be reduced to show that
	\begin{equation}\label{IntGoal}
		 \|f\|_{\mathfrak{B}^{n/p}_{p, n}} \lesssim  \Big(1-\frac{n}{p} \Big)^{1/n} \|f\|_{(L^p, \dot{W}^1_p)_{n/p, n}}  \qquad \text{as} \qquad p \to n+.
	\end{equation}
	
	 Recall that (note that $p \geq 2$) (cf. \cite[Proposition 2(iii), p. 47]{Triebel})
	$$
		L^p \hookrightarrow \mathfrak{B}^0_{p, p} \qquad \text{and} \qquad \dot{W}^1_p \hookrightarrow \mathfrak{B}^1_{p, p}
	$$
	with embedding constants independent of $p$ as $p \to n+$, and thus (cf. \eqref{DefInt})
	\begin{equation}\label{PBBM5}
		 \|f\|_{(\mathfrak{B}^0_{p, p}, \mathfrak{B}^1_{p, p})_{n/p, n}} \lesssim \|f\|_{(L^p, \dot{W}^1_p)_{n/p, n}}.
	\end{equation}
	Furthermore, $\mathfrak{B}^0_{p, p}$ and $\mathfrak{B}^1_{p, p}$  are isomorphic (with corresponding constants again independent of $p$) to $\ell^p (\Z \times \Z^n) =\ell^p$ and $\ell^p(\Z \times \Z^n; 2^j)=\ell^p(2^j)$, respectively,  via the wavelet transform
	\begin{equation}\label{WavTrans}
	f \mapsto \{2^{j n (\frac{1}{2}-\frac{1}{p})} \lambda_{j m}(f)\}.
	\end{equation}
	Here  $$\|\lambda\|_{\ell^p (2^{j \eta})} = \bigg(\sum_{j=-\infty}^\infty \sum_{m\in \Z^n} 2^{j \eta p } |\lambda_{j m}|^p\bigg)^{\frac{1}{p}}, \qquad \lambda = \{\lambda_{j m}\}, \qquad \eta = 0, 1.$$
 Therefore computability of the space $(\mathfrak{B}^0_{p, p}, \mathfrak{B}^1_{p, p})_{n/p, n}$ turns out to be equivalent to $(\ell^p, \ell^p (2^j))_{n/p, n}$. More precisely, we have
	\begin{equation}\label{PBBM6}
		\|f\|_{(\mathfrak{B}^0_{p, p}, \mathfrak{B}^1_{p, p})_{n/p, n}} \approx \|\{2^{j n (\frac{1}{2}-\frac{1}{p})} \lambda_{j m}(f)\}\|_{(\ell^p, \ell^p (2^j))_{n/p, n}},
	\end{equation}
	where the equivalence constants are independent of $p$.
	
	Next we estimate the $K$-functional for $(\ell^p, \ell^p (2^j))$. Let $t >0$ and $\lambda \in \ell^p + \ell^p(2^j)$. Since $p > 1$,  we have
	\begin{align}
		K(t, \lambda; \ell^p, \ell^p(2^j))^p &\geq \inf_{\lambda = \lambda^0 + \lambda^1} (\|\lambda^0\|^p_{\ell^p} + t^p \|\lambda^1\|^p_{\ell^p(2^j)}) \nonumber \\
		& \geq  \sum_{j=-\infty}^{\infty} \sum_{m \in \Z^n} \inf_{\lambda_{j m} = \lambda^0_{j m} + \lambda^1_{j m}}  (|\lambda_{j m}^0|^p + t^p 2^{j p} |\lambda^1_{j m}|^p)  \nonumber  \\
		& \geq 2^{-p} \sum_{j=-\infty}^{\infty} \sum_{m \in \Z^n} \Big[ \inf_{\lambda_{j m} = \lambda^0_{j m} + \lambda^1_{j m}}  (|\lambda_{j m}^0| + t 2^{j} |\lambda^1_{j m}|) \Big]^p. \label{NK1}
	\end{align}
	Furthermore, elementary computations lead to the minimization of the following variational problem
	$$
		\inf_{x = x^0 + x^1} (|x^0| + t |x^1|) = \min\{1, t\} |x|, \qquad x \in \R.
	$$
	Applying the latter formula (with $x = \lambda_{j m}$)  into \eqref{NK1} we get
	\begin{equation}\label{NK2}
		K(t, \lambda; \ell^p, \ell^p(2^j)) \geq 2^{-1} \bigg( \sum_{j=-\infty}^{\infty} \sum_{m \in \Z^n} (\min\{1, t 2^j\} |\lambda_{j m}|)^p  \bigg)^{\frac{1}{p}}.
	\end{equation}

	In light of \eqref{NK2}, the interpolation norm $\|\lambda\|_{(\ell^p, \ell^p(2^j ))_{n/p, n}}$ can be estimated as follows:
	\begin{align}
		\|\lambda\|_{(\ell^p, \ell^p(2^j ))_{n/p, n}} & = \bigg(\int_0^\infty [t^{-n/p} K(t, \lambda; \ell^p, \ell^p(2^j))]^n \, \frac{dt}{t} \bigg)^{1/n} \nonumber \\
		& \approx \bigg(\sum_{l=-\infty}^\infty [2^{l n/p} K(2^{-l}, \lambda; \ell^p, \ell^p (2^j))]^n \bigg)^{1/n} \nonumber \\
		& \approx \bigg(\sum_{l=-\infty}^\infty \bigg[2^{l n} \sum_{j=-\infty}^{\infty} \sum_{m \in \Z^n} (\min\{1,  2^{j-l}\} |\lambda_{j m}|)^p \bigg]^{n/p} \bigg)^{1/n} \nonumber \\
		& \approx \bigg(\sum_{l=-\infty}^\infty \bigg[2^{l (n-p)} \sum_{j=-\infty}^{l} \sum_{m \in \Z^n} 2^{j p} |\lambda_{j m}|^p \bigg]^{n/p} \bigg)^{1/n} \nonumber \\
		& \hspace{1cm} +  \bigg(\sum_{l=-\infty}^\infty \bigg[2^{l n}  \sum_{j =l+1}^\infty \sum_{m \in \Z^n}  |\lambda_{j m}|^p \bigg]^{n/p} \bigg)^{1/n} \nonumber\\
		& = : \mathcal{I}_1 + \mathcal{I}_2. \label{PBBM7}
	\end{align}
		It is clear that
	\begin{equation}\label{PBBM8}
		\mathcal{I}_2 \geq 2^{-n/p}  \left(\sum_{l=-\infty}^\infty \bigg[2^{l n}  \sum_{m \in \Z^n} |\lambda_{l m}|^p \bigg]^{n/p} \right)^{1/n} \geq 2^{-1} \|\lambda\|_{\ell^n  (2^{l n/p} \ell^p)},
	\end{equation}
	where
	$$
	\|\lambda\|_{\ell^n  (2^{l n/p} \ell^p)} = \bigg(\sum_{l=-\infty}^\infty 2^{\frac{l n^2}{p}} \bigg(\sum_{m \in \Z^n} |\lambda_{l m}|^p \bigg)^{\frac{n}{p}}  \bigg)^{\frac{1}{n}}.
	$$
	On the other hand, it is also obvious that
	$$
		\mathcal{I}_1 \geq \left(\sum_{l=-\infty}^\infty \bigg[2^{l n}  \sum_{m \in \Z^n} |\lambda_{l m}|^p \bigg]^{n/p} \right)^{1/n} = \|\lambda\|_{\ell^n (2^{l n/p} \ell^p)}.
	$$
	However, the latter estimate is very rough and can be improved to say that
	\begin{equation}\label{PBBM9}
		\mathcal{I}_1 \geq \bigg(1-\frac{n}{p} \bigg)^{-1/p} \|\lambda\|_{\ell^n  (2^{l n/p} \ell^p)}.
	\end{equation}
	 Note that $(1-\frac{n}{p})^{-\frac{1}{p}} \to \infty$ as $p \to n+$. Let us show \eqref{PBBM9}: Setting
	$$
		A_\nu = \sum_{j=-\infty}^{\nu} 2^{j p} \sum_{m \in \Z^n} |\lambda_{j m}|^p, \qquad \nu \in \Z,
	$$
	we have (recall that $p > n$)
	\begin{equation*}
		A_\nu^{\frac{n}{p}}-A_{\nu-1}^{\frac{n}{p}} = \frac{p}{n}  \int_{A_{\nu-1}}^{A_\nu} t^{\frac{n}{p}} \, \frac{dt}{t} \geq \frac{p}{n} A_\nu^{\frac{n}{p} -1} (A_\nu-A_{\nu-1})  = \frac{p}{n} A_\nu^{\frac{n}{p} -1}  2^{\nu p} \sum_{m \in \Z^n} |\lambda_{\nu m}|^p.
	\end{equation*}
	Summing up over $\nu \leq l$, we arrive at
	$$
		 \bigg(\sum_{j=-\infty}^{l} 2^{j p}  \sum_{m \in \Z^n} |\lambda_{j m}|^p \bigg)^{\frac{n}{p}}  \geq \frac{p}{n} \sum_{\nu=-\infty}^l  \bigg(\sum_{j=-\infty}^{\nu} 2^{j p}  \sum_{m \in \Z^n} |\lambda_{j m}|^p \bigg)^{\frac{n}{p} -1}  2^{\nu p}  \sum_{m \in \Z^n} |\lambda_{\nu m}|^p.
	$$
	Accordingly, changing the order of summation, we obtain
	\begin{align*}
		\mathcal{I}_1 & \geq \bigg(\sum_{l=-\infty}^\infty 2^{\frac{l (n-p)n}{p}} \sum_{\nu=-\infty}^l  \bigg(\sum_{j=-\infty}^{\nu} 2^{j p}  \sum_{m \in \Z^n} |\lambda_{j m}|^p \bigg)^{\frac{n}{p} -1}  2^{\nu p}  \sum_{m \in \Z^n} |\lambda_{\nu m}|^p \bigg)^{\frac{1}{n}}  \\
		& = \bigg(\sum_{\nu=-\infty}^\infty \bigg(\sum_{j=-\infty}^{\nu} 2^{j p}  \sum_{m \in \Z^n} |\lambda_{j m}|^p \bigg)^{\frac{n}{p} -1}  2^{\nu p}  \sum_{m \in \Z^n} |\lambda_{\nu m}|^p  \sum_{l=\nu}^\infty 2^{\frac{l (n-p)n}{p}}  \bigg)^{\frac{1}{n}} \\
		& \approx \bigg(\frac{p}{(p-n)n} \bigg)^{\frac{1}{n}}  \bigg(\sum_{\nu=-\infty}^\infty \bigg(\sum_{j=-\infty}^{\nu} 2^{j p}  \sum_{m \in \Z^n} |\lambda_{j m}|^p \bigg)^{\frac{n}{p} -1}  2^{\nu p}  \sum_{m \in \Z^n} |\lambda_{\nu m}|^p  2^{\frac{\nu(n-p) n}{p}} \bigg)^{\frac{1}{n}},
		\end{align*}
		where we have also used that $p > n$ in the last step. If we  apply now reverse H\"older's inequality with exponent $n/p < 1$, then
		\begin{align*}
		\mathcal{I}_1&  \gtrsim \bigg(\frac{p}{(p-n)n} \bigg)^{\frac{1}{n}} \left(\sum_{\nu=-\infty}^\infty \bigg[2^{\nu n}  \sum_{m \in \Z^n} |\lambda_{\nu m}|^p\bigg]^{\frac{n}{p}} \right)^{\frac{p}{n^2}} \left(\sum_{\nu=-\infty}^\infty 2^{\frac{\nu(n-p) n}{p}}   \bigg[\sum_{j=-\infty}^{\nu} 2^{j p}  \sum_{m \in \Z^n} |\lambda_{j m}|^p \bigg]^{\frac{n}{p}}   \right)^{\frac{n-p}{n^2}} \\
		& = \bigg(\frac{p}{(p-n)n} \bigg)^{\frac{1}{n}} \left(\sum_{\nu=-\infty}^\infty \bigg[2^{\nu n}  \sum_{m \in \Z^n} |\lambda_{\nu m}|^p\bigg]^{\frac{n}{p}} \right)^{\frac{p}{n^2}} \mathcal{I}_1^{\frac{n-p}{n}}.
	\end{align*}
	Hence
	$$
		\mathcal{I}_1 \gtrsim  \bigg(\frac{p}{(p-n)n} \bigg)^{\frac{1}{p}} \left(\sum_{\nu=-\infty}^\infty \bigg[2^{\nu n}  \sum_{m \in \Z^n} |\lambda_{\nu m}|^p\bigg]^{\frac{n}{p}} \right)^{\frac{1}{n}},
	$$
	i.e., \eqref{PBBM9} holds.
	
	Inserting \eqref{PBBM8} and \eqref{PBBM9} into \eqref{PBBM7}, we get
	$$
		\|\lambda\|_{(\ell^p , \ell^p(2^j))_{n/p, n}}  \gtrsim \bigg(1-\frac{n}{p} \bigg)^{-1/p}  \|\lambda\|_{\ell^n  (2^{j n/p} \ell^p)}.
	$$
	Using the fact that $\mathfrak{B}^{n/p}_{p, n}$ is isomorphic to $\ell^n  (2^{j n/p} \ell^p)$ through \eqref{WavTrans} (with equivalence constants independent of $p$), from \eqref{PBBM6} we achieve that
	\begin{equation}\label{PBBM10}
		\|f\|_{(\mathfrak{B}^0_{p, p}, \mathfrak{B}^1_{p, p})_{n/p, n}} \gtrsim \bigg(1-\frac{n}{p} \bigg)^{-1/p}  \|f\|_{\mathfrak{B}^{n/p}_{p, n}}.
	\end{equation}
	Then, by  \eqref{PBBM5} and \eqref{PBBM10},
	$$
		\bigg(1-\frac{n}{p} \bigg)^{-1/p}  \|f\|_{\mathfrak{B}^{n/p}_{p, n}} \lesssim \|f\|_{(L^p, \dot{W}^1_p)_{n/p, n}}.
	$$
	Noting that
	$$
		\bigg(1-\frac{n}{p} \bigg)^{1/p} \approx \bigg(1-\frac{n}{p} \bigg)^{1/n} \qquad \text{as} \qquad p \to n+,
	$$
	 the proof of \eqref{IntGoal} (and hence \eqref{PBBM3}) is finished.
\end{proof}

The applicability of our method requires sharp estimate for the Besov norm given in the right-hand side of \eqref{PropBBM} in terms of standard Sobolev norms. This can be accomplished via the well-known Franke--Jawerth embeddings (see e.g. \cite[(10.2), p. 83]{Kolyada})
$$
	\dot{W}^1_n \hookrightarrow B^{n/p}_{p, n}, \qquad 2 \leq n < p < \infty.
$$
That is,  there exists $c_{n, p} >0$, depending on $n$ and $p$, such that
$$
	\|f\|_{B^{n/p}_{p, n}} \leq c_{n, p} \,  \|\nabla f\|_{L^n}.
$$
However, for our purposes, this inequality is not strong enough  and the optimal blow-up of $c_{n, p}$ as $p \to n+$ is needed\footnote{$\lim_{p \to n +} \|f\|_{B^{n/p}_{p, n}} < \infty \implies f \equiv \text{constant}$.}. This is the content of the following result.

\begin{thm}\label{ThmSharpIlin}
	Let $2 \leq n < p <p_0 < \infty$. Then there exists a constant $c_{n, p_0}$, depending only on $n$ and $p_0$, such that
	$$
		 \|f\|_{B^{n/p}_{p, n}} \leq c_{n, p_0}  \Big(1-\frac{n}{p} \Big)^{-1/n} \,  \|\nabla f\|_{L^n}, \qquad \forall f \in \dot{W}^{1}_n.
	$$
\end{thm}

\begin{proof}
	We first establish the following estimate of moduli of smoothness in terms of rearrangements. Let $n > 1, \frac{n}{n-1} \leq p < \infty$, and $r = \frac{n p}{n+p}$,  then (cf. \eqref{EquivKfunctMod})
		\begin{equation}\label{4}
		\tilde{\omega}(f, t)_p \leq C_{n, p} \left\{ \bigg(\int_0^{t^n}  (|\nabla f|^*(u))^r \, du \bigg)^{\frac{1}{r}} + t \bigg(\int_{t^n}^\infty (|\nabla f|^*(u))^p \, du \bigg)^{\frac{1}{p}} \right\}
	\end{equation}
	for some constant $C_{n, p} \approx 1$ as $p \to n+$. The proof of \eqref{4} can be done via standard interpolation techniques. For the convenience of the reader, we sketch below the proof.
	
	 Applying Sobolev's inequality (note that $1 \leq r < n$):
	\begin{equation}\label{3}
		\dot{W}^{1}_{r} \hookrightarrow L^p \qquad \text{with} \qquad \|f\|_{L^p} \leq S_{n, p} \, \|\nabla f\|_{L^r},
	\end{equation}
	where the Sobolev constant $S_{n, p} \to \infty$ as $p \to \infty$ (i.e. $r \to n-$), but $S_{n, p} \approx 1$ as $p \to n+$ (i.e. $r \to \big( \frac{n}{2} \big)_+$). If we interpolate \eqref{3} with the trivial embedding $\dot{W}^1_p \hookrightarrow \dot{W}^1_p$, it is plain to see that (cf. \eqref{DefKFunct})
	\begin{equation}\label{C1}
		K(S_{n, p} t, f; L^p, \dot{W}^1_p) \leq S_{n, p} \, K ( t, f; \dot{W}^{1}_r, \dot{W}^1_p), \qquad \forall t > 0.
	\end{equation}
	The left-hand side can be computed using \eqref{EquivKfunctMod}:
	\begin{equation}\label{C2}
		K(S_{n, p} t, f; L^p, \dot{W}^1_p) \approx \tilde{\omega} (f, S_{n, p} t)_p \approx  \tilde{\omega} (f, t)_p  \qquad \text{as} \qquad p \to n+.
	\end{equation}
	On the other hand, characterizations of the $K$-functional appearing in the right-hand side of \eqref{C1} follows from the work of DeVore and Scherer \cite{DeVoreScherer}, namely,
	\begin{equation}\label{C3}
		K ( t, f; \dot{W}^{1}_r, \dot{W}^1_p) \approx  \bigg(\int_0^{t^n}  (|\nabla f|^*(u))^r \, du \bigg)^{\frac{1}{r}} + t \bigg(\int_{t^n}^\infty (|\nabla f|^*(u))^p \, du \bigg)^{\frac{1}{p}},
	\end{equation}
	with underlying equivalence constant independent of $p$. Collecting now \eqref{C1}--\eqref{C3}, we obtain the desired estimate \eqref{4}.

	Since $\omega(f, t)_p \leq \tilde{\omega}(f, t)_p$ (cf. \eqref{EquivMod}), we can apply \eqref{4} together with a simple change of variables and basic monotonicity properties to get
	\begin{align}
		\|f\|_{B^{n/p}_{p, n}} & \leq  \bigg(\int_0^\infty (t^{-n/p} \tilde{\omega}(f, t)_p)^n \, \frac{dt}{t} \bigg)^{1/n} 	\nonumber \\
		& \lesssim \left(\int_0^\infty t^{-n/p}  \bigg(\int_0^{t}  (|\nabla f|^*(u))^r \, du \bigg)^{\frac{n}{r}} \, \frac{dt}{t}  \right)^{1/n} + \left(\int_0^\infty t^{1-n/p} \bigg(\int_{t}^\infty  (|\nabla f|^*(u))^p \, du \bigg)^{\frac{n}{p}} \, \frac{dt}{t}   \right)^{1/n} \nonumber  \\
		&\approx \left(\int_0^\infty t^{-n/p}  \bigg(\int_0^{t}  (|\nabla f|^*(u))^r \, du \bigg)^{\frac{n}{r}} \, \frac{dt}{t}  \right)^{1/n} + \left(\sum_{j=-\infty}^\infty 2^{-j(1-n/p)} \bigg(\sum_{l=-\infty}^j  (|\nabla f|^*(2^{-l}))^p 2^{-l} \bigg)^{\frac{n}{p}}    \right)^{1/n} \nonumber  \\
		& =: \mathcal{I}_1 + \mathcal{I}_2. \label{C4}
		\end{align}
		
		To estimate $\mathcal{I}_1$, we can invoke the sharp version of Hardy's inequality as stated in \cite[Lemma 2.1]{Kolyada},
		\begin{align}
		\mathcal{I}_1& \leq \bigg(\frac{p}{r} \bigg)^{1/r} \left(\int_0^\infty [t  (|\nabla f|^*(t))^r]^{n/r}  t^{-n/p} \, \frac{dt}{t}  \right)^{1/n} \nonumber  \\
		& = \bigg(1 + \frac{p}{n} \bigg)^{\frac{1}{p}+\frac{1}{n}} \left(\int_0^\infty (|\nabla f|^*(t))^n \, dt \right)^{1/n}  \approx  \left(\int_0^\infty (|\nabla f|^*(t))^n \, dt \right)^{1/n}   =   \|\nabla f\|_{L^n}. \label{C5}
	\end{align}
	On the other hand, $\mathcal{I}_2$ can be estimated as follows: use that $p > n$ and change the order of summation so that
	\begin{align}
		\mathcal{I}_2 &\leq  \left(\sum_{j=-\infty}^\infty 2^{-j(1-n/p)} \sum_{l=-\infty}^j  (|\nabla f|^*(2^{-l}))^n 2^{-l n/p}    \right)^{1/n} \nonumber \\
		& \leq  \left(\sum_{l=-\infty}^\infty  (|\nabla f|^*(2^{-l}))^n 2^{-l n/p}  \sum_{j= l }^\infty 2^{-j(1-n/p)}    \right)^{1/n}  \nonumber \\
		& \approx \bigg(1-\frac{n}{p} \bigg)^{-1/n}  \left(\sum_{l=-\infty}^\infty  (|\nabla f|^*(2^{-l}))^n   2^{-l}    \right)^{1/n} \nonumber \\
		& \approx \bigg(1-\frac{n}{p} \bigg)^{-1/n}   \|\nabla f\|_{L^n}. \label{C6}
	\end{align}
	
	Plugging \eqref{C5} and \eqref{C6} into \eqref{C4}, we arrive at
	$$
		\|f\|_{B^{n/p}_{p, n}}  \lesssim  \bigg(1-\frac{n}{p} \bigg)^{-1/n}   \|\nabla f\|_{L^n}
	$$
	and the proof is finished.
\end{proof}

\subsection{Sharp fractional Poincar\'e inequalities}
Recall the directional Poincar\'e inequalities recently obtained by Haddad, Jim\'enez, and Montenegro in \cite[Lemma 1]{HaddadJimenezMontenegro}.

\begin{thm}\label{ThmHLM}
	Let $\Omega \subset \R^n$ be a bounded open set  and let $p \in [1, \infty)$.  Then, for every $\xi \in \S^{n-1}$ and every  $f \in C^1$ with support in $\Omega$,
		\begin{equation}\label{PoinAff}
	   \|f\|_{L^p} \leq c_p^{-1} \text{\rm{w}}(\Omega, \xi)  \, \|\nabla f \cdot \xi\|_{L^p},
	\end{equation}
	where $\rm{w} (\Omega, \xi)$ is the width of $\Omega$ in the direction of $\xi$ and the constant $c_p$ depends\footnote{$c_p = 	 \left\{\begin{array}{cl}  \frac{2}{p} \frac{\pi-\frac{\pi}{p}}{\sin \frac{\pi}{p}} (p-1)^{\frac{1}{p}-1}, & \text{if} \quad p > 1, \\
	2,  & \text{if} \quad p=1.
		       \end{array}
                        \right.$} only on $p$.
\end{thm}

Our next result consists of an improvement of Theorem \ref{ThmHLM}; cf. Remark  \ref{RemHLM} below.

\begin{thm}\label{ThmOptim2}
Let $\Omega \subset \R^n$ be a bounded open set  and let $s \in (0, 1), p, q \in [1, \infty)$.  Then, for every $\xi \in \S^{n-1}$ and every  $f \in C^1$ with support in $\Omega$,
\begin{equation}\label{425}
	\|f\|_{L^p} \leq C  (1-s)^{\frac{1}{q}} \text{\rm{w}} (\Omega, \xi)^{s} \left( \int_0^{\text{\rm{w}}(\Omega, \xi)} t^{- s q} \|\Delta_{t \xi} f\|_{L^p}^q \, \frac{dt}{t} \right)^{\frac{1}{q}},
	\end{equation}
	where  $C$ can be taken to be\footnote{Note that $C \approx C_{p, q}$ for $s \in (0, 1)$, where $C_{p, q}$ denotes a constant that depends only on $p$ and $q$.}
	$$
		C = \frac{9}{c_p} \bigg(\frac{q}{s q + 1} \bigg)^{\frac{1}{q}} \max \bigg\{1, \frac{c_p}{3} \bigg\}^{1-s}
	$$
	with $c_p$ being the constant in Theorem \ref{ThmHLM}.
\end{thm}

\begin{rem}\label{RemHLM}
	Using that $\|\Delta_{t \xi} f\|_{L^p} \leq t \|\nabla f \cdot \xi\|_{L^p}$ (cf. \eqref{314b2}), we have
	$$
		((1-s)q)^{\frac{1}{q}} \text{\rm{w}} (\Omega, \xi)^{s} \left( \int_0^{\text{\rm{w}}(\Omega, \xi)} t^{- s q} \|\Delta_{t \xi} f\|_{L^p}^q \, \frac{dt}{t} \right)^{\frac{1}{q}} \leq \text{\rm{w}}(\Omega, \xi) \, \|\nabla f \cdot \xi\|_{L^p},
	$$
	i.e., the right-hand side of \eqref{425} can always be estimated by the corresponding one of \eqref{PoinAff}. In addition, the following limit formula holds
	$$
		\lim_{s \to 1-} \text{\rm{w}} (\Omega, \xi)^{s} (1-s)^{\frac{1}{q}}  \left( \int_0^\infty t^{- s q} \|\Delta_{t \xi} f\|_{L^p}^q \, \frac{dt}{t} \right)^{\frac{1}{q}} = q^{-\frac{1}{q}}  \text{\rm{w}}(\Omega, \xi) \, \|\nabla f \cdot \xi\|_{L^p}.
	$$
	The proof of this result can be obtained similarly as \eqref{Th1-e1}.
\end{rem}

\begin{rem}
	One can not expect \eqref{425} to hold with the additional prefactor $s^{1/q}$ on the right-hand side, that is, the Maz'ya--Shaposhnikova phenomenon does not appear in this case. To see this, one may use \eqref{314b2} and take limits as $s \to 0+$.
\end{rem}

\begin{proof}[Proof of Theorem \ref{ThmOptim2}]
Let $\Omega_\xi = \{x \in \R^n : \text{dist } (x, \overline{\Omega}) <  \text{\rm{w}} (\Omega, \xi)\}$. Note that $\Omega_\xi$ is a bounded  open set with $\Omega \subset \Omega_\xi$ and
\begin{equation}\label{WidthDomain}
	\text{\rm{w}}(\Omega_\xi, \xi) \leq 3 \text{\rm{w}} (\Omega, \xi).
\end{equation}

For every fixed $t \in (0, w(\Omega, \xi))$, define
$$
	g_{ t} (x) = \frac{1}{t} \int_0^t f(x + u \xi) \, du, \qquad x \in \R^n.
$$
We next show that
\begin{equation}\label{423n}
\text{supp } g_{ t} \subset \Omega_\xi.
\end{equation}
 Indeed, we will prove that $x \not \in \Omega_\xi$ guarantees $x + u \xi \not \in \Omega$ for all $u \in (0, \text{\rm{w}}(\Omega, \xi))$ (and hence $x \not \in \text{supp } g_{t}$). Otherwise, if there exists $u \in (0, \text{\rm{w}}(\Omega, \xi))$ such that $x + u \xi \in \Omega$, then
$
	|x-(x+u \xi)| = u < \text{\rm{w}}(\Omega, \xi).
$
In particular, $\text{dist }(x, \overline{\Omega}) <  \text{\rm{w}}(\Omega, \xi)$ and then $x \in \Omega_\xi$, which is not true. This proves the desired claim \eqref{423n}.

We wish to apply \eqref{PoinAff} to $f = g_{t}$ and $\Omega = \Omega_\xi$. Accordingly
\begin{equation*}
	c_p \text{w} (\Omega_\xi, \xi)^{-1}  \|g_t\|_{L^p} \leq \|\nabla g_t \cdot \xi\|_{L^p},
	\end{equation*}
	and, by \eqref{WidthDomain},
	\begin{equation}\label{PoinAffv2}
	\frac{c_p}{3} \text{w} (\Omega, \xi)^{-1}  \|g_t\|_{L^p} \leq \|\nabla g_t \cdot \xi\|_{L^p}.
	\end{equation}
	Furthermore, basic computations yield that
	$$
		\nabla g_t(x) \cdot \xi = \frac{1}{t} \int_0^t \nabla f(x + u \xi) \cdot \xi \, du = \frac{1}{t} \int_0^t \frac{d}{d u} f(x + u \xi) \, d u = \frac{\Delta_{t \xi} f (x)}{t}
	$$
	and hence
	\begin{equation*}
		 \|\nabla g_t \cdot \xi\|_{L^p} = \frac{\|\Delta_{t \xi} f\|_{L^p}}{t}.
	\end{equation*}
	Plugging this into \eqref{PoinAffv2}, we arrive at
	\begin{equation}\label{PoinAffv2n}
		\frac{c_p}{3} \text{w} (\Omega, \xi)^{-1}  \|g_t\|_{L^p} \leq  \frac{\|\Delta_{t \xi} f\|_{L^p}}{t}.
	\end{equation}
	
	On the other hand, we can estimate the $L^p$ norm of $f-g_t$ via integral Minkowski's inequality as follows:
	\begin{equation}\label{PoinAffv3n}
		\|f-g_t\|_{L^p}  = \frac{1}{t} \bigg\| \int_0^t (f(\cdot)-f(\cdot+u \xi)) \, du \bigg\|_{L^p} \leq \frac{1}{t} \int_0^t \|\Delta_{u \xi} f\|_{L^p} \, du.
	\end{equation}
	As a combination of \eqref{PoinAffv2n} and \eqref{PoinAffv3n}, we get
	$$
		\|f-g_t\|_{L^p}  + \frac{c_p}{3} \text{w} (\Omega, \xi)^{-1} t  \|g_t\|_{L^p} \leq \frac{1}{t} \int_0^t \|\Delta_{u \xi} f\|_{L^p} \, du + \|\Delta_{t \xi} f\|_{L^p}.
	$$
	Therefore
	\begin{align}
		\min \bigg\{1, \frac{c_p}{3} \text{w} (\Omega, \xi)^{-1} t \bigg\} \|f\|_{L^p} & \leq \min \bigg\{1, \frac{c_p}{3} \text{w} (\Omega, \xi)^{-1} t \bigg\}  (\|f-g_t\|_{L^p} + \|g_t\|_{L^p}) \nonumber \\
		& \leq \|f-g_t\|_{L^p} +  \frac{c_p}{3} \text{w} (\Omega, \xi)^{-1} t  \|g_t\|_{L^p} \nonumber \\
		& \leq \frac{1}{t} \int_0^t \|\Delta_{u \xi} f\|_{L^p} \, du + \|\Delta_{t \xi} f\|_{L^p}.\label{PoinAffv4n}
	\end{align}
	
	Since $\Delta_{t \xi} f (x) = \Delta_{(t-u) \xi} f(x+u \xi) + \Delta_{u \xi} f(x), \, u \in (0, t)$, we obtain (after a simple change of variables)
	$$
		\|\Delta_{t \xi} f\|_{L^p} \leq \| \Delta_{(t-u) \xi} f\|_{L^p} + \|\Delta_{u \xi} f\|_{L^p}.
	$$
	Integrating the last expression over all $u \in (0, t)$, we derive
	\begin{equation*}
		t \|\Delta_{t \xi} f\|_{L^p}  \leq \int_0^t \| \Delta_{(t-u) \xi} f\|_{L^p}  \, du + \int_0^t \|\Delta_{u \xi} f\|_{L^p} \, du  = 2 \int_0^t \|\Delta_{u \xi} f\|_{L^p} \, du,
	\end{equation*}
	where we have applied another change of variables in the last step. Hence, by \eqref{PoinAffv4n},
	\begin{equation*}
		\min \bigg\{1, \frac{c_p}{3} \text{w} (\Omega, \xi)^{-1} t \bigg\} \|f\|_{L^p}  \leq \frac{3}{t} \int_0^t   \|\Delta_{u \xi} f\|_{L^p} \, du
	\end{equation*}
	and, using H\"older's inequality with exponent $q$,
	\begin{equation}\label{PoinAffv5n}
	\min \bigg\{1, \frac{c_p}{3} \text{w} (\Omega, \xi)^{-1} t \bigg\} \|f\|_{L^p}  \leq 3 \bigg(\frac{1}{t}  \int_0^t   \|\Delta_{u \xi} f\|^q_{L^p} \, du \bigg)^{\frac{1}{q}}
	\end{equation}
	for every $t \in (0, \text{w}(\Omega, \xi))$.

	If we multiply both sides of \eqref{PoinAffv5n} by $t^{-s}$ and then apply $L^q((0, \text{w}(\Omega, \xi); \frac{d t}{t})$ norms, we achieve
	\begin{equation}\label{PoinAffv6n}
	\bigg(\int_0^{\text{w}(\Omega, \xi)} t^{-s q} \min \bigg\{1, \frac{c_p}{3} \text{w} (\Omega, \xi)^{-1} t \bigg\}^q \, \frac{dt}{t}   \bigg)^{\frac{1}{q}} \, \|f\|_{L^p} \leq	3 \bigg(\int_0^{\text{w}(\Omega, \xi)} t^{-s q-1}  \int_0^t   \|\Delta_{u \xi} f\|^q_{L^p} \, du \, \frac{d t}{t} \bigg)^{\frac{1}{q}}.
	\end{equation}
	On the one hand, the right-hand side can be computed by using Fubini's theorem:
	\begin{align}
		\int_0^{\text{w}(\Omega, \xi)} t^{-s q-1}  \int_0^t   \|\Delta_{u \xi} f\|^q_{L^p} \, du \, \frac{d t}{t} &= \int_0^{\text{w}(\Omega, \xi)} \|\Delta_{u \xi} f\|^q_{L^p} \int_{u}^{\text{w}(\Omega, \xi)} t^{-sq -1} \, \frac{dt}{t} \, du \nonumber  \\
		& \leq \frac{1}{s q + 1}  \int_0^{\text{w}(\Omega, \xi)} u^{-s q} \|\Delta_{u \xi} f\|^q_{L^p} \, \frac{du}{u}. \label{PoinAffv7n}
	\end{align}
	On the other hand, to estimate the left-hand side of \eqref{PoinAffv6n} one can proceed as follows
	\begin{align}
	\int_0^{\text{w}(\Omega, \xi)} t^{-s q} \min \bigg\{1, \frac{c_p}{3} \text{w} (\Omega, \xi)^{-1} t \bigg\}^q \, \frac{dt}{t} &\geq  \bigg( \frac{c_p}{3} \text{w} (\Omega, \xi)^{-1}\bigg)^q \int_0^{\min\{1, \frac{3}{c_p}\}\text{w}(\Omega, \xi)} t^{(1-s) q} \, \frac{dt}{t}  \nonumber  \\
	& \hspace{-4cm}= \bigg( \frac{c_p}{3} \bigg)^q ((1-s)q)^{-1}  \min \bigg\{1, \frac{3}{c_p} \bigg\}^{(1-s) q} \text{w}(\Omega, \xi)^{-s q}.  \label{PoinAffv8n}
	\end{align}
	Finally, \eqref{PoinAffv6n}, \eqref{PoinAffv7n} and \eqref{PoinAffv8n} result in
	$$
		   \|f\|_{L^p} \leq \frac{9}{c_p} \bigg(\frac{(1-s) q}{s q + 1} \bigg)^{\frac{1}{q}} \max \bigg\{1, \frac{c_p}{3} \bigg\}^{1-s}  \text{w}(\Omega, \xi)^{s} \, \bigg( \int_0^{\text{w}(\Omega, \xi)} u^{-s q} \|\Delta_{u \xi} f\|^q_{L^p} \, \frac{du}{u} \bigg)^{\frac{1}{q}}.
	$$
\end{proof}

%%%%%%%%%%%%%%%%%%
\section{Affine fractional Moser--Trudinger inequalities}\label{Section5}

%%%%%%%%%%%%%%%%%%%

%\begin{thm}\label{ThmFractCLYY}
%	Let $2 \leq n < p < \infty$. Then there exists a constant $c_n$, depending only on $n$, such that for every $\beta > c_n$ we have
%	$$
%		\frac{1}{|\emph{supp } f|}	\int_{\R^n} \bigg[\exp \bigg(\frac{|f(x)|}{\beta \, \omega_n^{1/p} (1-\frac{n}{p})^{1/n}  |\Pi^{*, n/p}_{p, n} f|^{-1/p}} \bigg)^{n'} - 1 \bigg] \, dx \leq A,
%	$$
%	where $A$ is a certain absolute constant.
%\end{thm}

\begin{proof}[Proof of Theorem \ref{ThmFractCLYYIntro}]
	By $c_n$ we denote a purely dimensional constant that may vary from line to line.  As a combination of \eqref{347} and \eqref{347*},  we derive
	\begin{equation}\label{PF1}
		 \omega_n^{-1/r} |\Pi^{*, n/r}_{r, n} f^\star|^{1/r} \|f^\star\|_{B^{n/r}_{r, n}}= \|f^\star\|_{B^{n/r}_{r, n; \Pi^{*, n/r}_{r, n} f^\star}} \leq  4 \bigg( \frac{2^n +1}{\frac{n}{r} + 1} \bigg)^{1/n}.
	\end{equation}
	On the other hand, from Theorem \ref{ThmEmbBesov1} with $f=f^\star$ (recall that $|\text{supp }f| = |\text{supp } f^\star|$),
	\begin{equation}\label{PF2*}
		\|f\|_{L^q} = \|f^\star\|_{L^q} \leq c_n |\text{supp } f|^{1/q}    q^{1/n'} \Big(1-\frac{n}{r} \Big)^{1/n} \|f^\star\|_{B^{n/r}_{r, n}}.
	\end{equation}
	Putting together \eqref{PF1} and \eqref{PF2*},
	$$
		\|f\|_{L^q(\R^n)} \leq c_n |\text{supp } f|^{1/q}    q^{1/n'} \Big(1-\frac{n}{r} \Big)^{1/n} 4 \bigg( \frac{2^n +1}{\frac{n}{r} + 1} \bigg)^{\frac{1}{n}}  \omega_n^{\frac{1}{r}} |\Pi^{*, n/r}_{r, n} f^\star|^{-\frac{1}{r}}.
	$$
	Note that $\big( \frac{2^n +1}{\frac{n}{r} + 1} \big)^{\frac{1}{n}} \leq 4$ and then the previous estimate gives
			\begin{equation}\label{PF2}
		\|f\|_{L^q(\R^n)} \leq c_n |\text{supp } f|^{1/q}    q^{1/n'} \Big(1-\frac{n}{r} \Big)^{1/n} \omega_n^{1/r}  |\Pi^{*, n/r}_{r, n} f^\star|^{-1/r}.
	\end{equation}
	
	Since $n < r$, we can apply Theorem \ref{ThmPSBod}:
	\begin{equation*}
		|\Pi^{*, n/r}_{r, n} f^\star|^{-1/r} \leq 4 (2^n + 1)^{1/n} |\Pi^{*, n/r}_{r, n} f|^{-1/r}.
	\end{equation*}
	Plugging this into \eqref{PF2}, we have
	\begin{equation}\label{PF3}
			\|f\|_{L^q(\R^n)} \leq c_n |\text{supp } f|^{1/q}     q^{1/n'} \Big(1-\frac{n}{r} \Big)^{1/n} \omega_n^{1/r}  |\Pi^{*, n/r}_{r, n} f|^{-1/r}.
	\end{equation}
	
	Let $\lambda > 1$. Using the series development of the exponential function, it follows from \eqref{PF3} with  $q = k n'$ that
	\begin{align*}
	\frac{1}{|\text{supp } f|}	\int_{\R^n} \bigg[ \exp \bigg(\frac{|f(x)|}{c_n (n' \lambda e)^{1/n'} \, \omega_n^{1/r} (1-\frac{n}{r})^{1/n}  |\Pi^{*, n/r}_{r, n} f|^{-1/r}} \bigg)^{n'} - 1 \bigg] \, dx & \\
	& \hspace{-9cm}= 	\frac{1}{|\text{supp } f|}	\,  \sum_{k=1}^\infty \frac{1}{k! (n' \lambda e)^k} \int_{\R^n} \bigg(\frac{|f(x)|}{c_n  \omega_n^{1/r} (1-\frac{n}{r})^{1/n}  |\Pi^{*, n/r}_{r, n} f|^{-1/r}} \bigg)^{k n'}  \, dx \\
	& \hspace{-9cm} \leq \sum_{k=1}^\infty \frac{1}{k!} \bigg(\frac{k}{\lambda e} \bigg)^k.
	\end{align*}
	Note that the last series is convergent using Stirling's formula
	$
		k! = \sqrt{2 \pi k} \big(\frac{k}{e} \big)^k \big(1 + O \big(\frac{1}{k} \big) \big).
	$
\end{proof}

\begin{proof}[Proof of Theorem \ref{Theorem5.6}]	
	According to the comparison result of Huang and Li \cite[Theorem 1.2]{HuangLi}: Given $f \in W^{1}_{n}$ there exists $T_n \in \text{SL}(n)$ and a constant $c_n$, which depends only on $n$,  satisfying
	\begin{equation}\label{ProofOptimAff1}
		c_n \,  \|\nabla (f \circ T_n)\|_{L^n} \leq |\Pi^*_n f|^{-1/n}.
	\end{equation}
	On the other hand, applying Theorem \ref{ThmSharpIlin} with $f \circ T_n$, we have
	\begin{equation}\label{ProofOptimAff2}
		 \bigg(1-\frac{n}{r} \bigg)^{1/n} \|f \circ T_n\|_{B^{n/r}_{r, n}} \lesssim \|\nabla (f \circ T_n)\|_{L^n}.
	\end{equation}
	 From \eqref{RelationB}, \eqref{ProofOptimAff2} and \eqref{ProofOptimAff1},
	 \begin{equation}\label{D1}
	 	 \mathcal{G}_r (f \circ T_n)\lesssim \|\nabla (f \circ T_n)\|_{L^n} \lesssim |\Pi^*_n f|^{-1/n}.
	 \end{equation}
	 Furthermore, in light of \eqref{ProofOptimAff0}, the left-hand side of \eqref{D1} coincides with
	 $
	 	 \mathcal{G}_r(f).
	 $
	 This completes the proof of \eqref{D2}.
	
	 Next we proceed with the proof of \eqref{LimP1}. Without loss of generality, we may assume that $r \in (n, 2n)$ and $\text{supp } f \subset \Omega = B(0, R)$ for some $R > 0$. It follows from  Theorem \ref{ThmOptim2} and H\"older's inequality that
\begin{align*}
	\left[\bigg(1-\frac{n}{r} \bigg) \int_0^\infty t^{-\frac{n^2}{r}} \|\Delta_{t \xi} f\|_{L^r}^n \, \frac{dt}{t} \right]^{\frac{r}{n}} &\gtrsim \text{w}(\Omega, \xi)^{-n}   \|f\|_{L^r}^r  \\
	& \hspace{-3.5cm} \geq  \text{w}(\Omega, \xi)^{-n} |\Omega|    \bigg( \frac{\|f\|_{L^n}}{|\Omega|^{\frac{1}{n}}} \bigg)^r  \geq \text{w}(\Omega, \xi)^{-n} |\Omega| \min\{1, A\}^n A^n,
\end{align*}
where $A= \frac{\|f\|_{L^n}}{|\Omega|^{1/n}}$. In particular
$$
	\inf_{r \in (n, 2 n)} \left[\bigg(1-\frac{n}{r} \bigg) \int_0^\infty t^{-\frac{n^2}{r}} \|\Delta_{t \xi} f\|_{L^r}^n \, \frac{dt}{t} \right]^{\frac{r}{n}}  \gtrsim \text{w}(\Omega, \xi)^{-n} |\Omega| \min\{1, A\}^n A^n,
$$
and hence
\begin{align}
\sup_{\xi \in \S^{n-1}}\sup_{r \in (n, 2 n)}   \bigg(1-\frac{n}{r} \bigg)^{-\frac{r}{n}}\|\xi\|^{-n}_{\Pi^{*, \frac{n}{r}}_{r, n} f} &=	\sup_{\xi \in \S^{n-1}}\sup_{r \in (n, 2 n)} \left[\bigg(1-\frac{n}{r} \bigg) \int_0^\infty t^{-\frac{n^2}{r}} \|\Delta_{t \xi} f\|_{L^r}^n \, \frac{dt}{t} \right]^{-\frac{r}{n}} \nonumber \\
& \lesssim  \frac{1}{|\Omega|} \bigg(\frac{\text{diam } \Omega}{\min\{1, A\} A} \bigg)^n.  \label{LimP2}
\end{align}

Since (cf. \eqref{Volume})
$$
	\mathcal{G}_r(f)^{-r} =  \frac{1}{n} \int_{\S^{n-1}} \bigg(1-\frac{n}{r} \bigg)^{-\frac{r}{n}}\|\xi\|^{-n}_{\Pi^{*, \frac{n}{r}}_{r, n} f} \, d \xi
$$
 and \eqref{LimP2} hold, we can apply the dominated convergence theorem together with Theorem \ref{Th1} to obtain
\begin{equation*}
		\lim_{r \to n+} \mathcal{G}_r(f)^{-r}= \frac{1}{n} \int_{\S^{n-1}} \lim_{r \to n+} \bigg(1-\frac{n}{r} \bigg)^{-\frac{r}{n}}\|\xi\|^{-n}_{\Pi^{*, \frac{n}{r}}_{r, n} f} \, d \xi  = \int_{\S^{n-1}} \|\xi\|^{-n}_{\Pi^*_n f} \, d \xi = n |\Pi^*_n f|.
\end{equation*}
\end{proof}

%The proof of Theorem \ref{ThmOptim} makes use of the following limiting formulas involving volumes of Besov and classical polar projection bodies.

%\begin{prop}\label{PropConVol}
%Let $f \in C^2_c(\R^n), \, n \geq 1$. Then
%\begin{equation}\label{LimP1}
%	\lim_{r \to n+} \mathcal{G}_r(f) = (n  |\Pi^*_n f|)^{-\frac{1}{n}}.
%\end{equation}
%\end{prop}

\begin{proof}[Proof of Theorem \ref{ThmOptim}]

The assertion (i) is a simple consequence of \eqref{D2}.

(ii): Let $x \in \R^n$. Then
\begin{equation}
	\sup_{r \in (n, 2n)} \left[  \exp \bigg(\frac{|f(x)|}{\beta \omega_n^{1/r} \mathcal{G}_r(f)} \bigg)^{n'} -1 \right] = \exp \bigg[\frac{|f(x)|}{\beta} \sup_{r \in (n, 2 n)} \omega_n^{-1/r} \mathcal{G}_r(f)^{-1}  \bigg]^{n'}-1. \label{LimP8}
\end{equation}
In view of \eqref{Volume} and \eqref{LimP2}, for $r \in (n, 2 n)$,
\begin{align*}
	 \mathcal{G}_r(f)^{-r}  & \leq \frac{1}{n} \int_{\S^{n-1}}  \sup_{r \in (n, 2 n)}  \bigg(1-\frac{n}{r} \bigg)^{-\frac{r}{n}}\|\xi\|^{-n}_{\Pi^{*, \frac{n}{r}}_{r, n} f} \, d \xi \\
	& \lesssim \sup_{\xi \in \S^{n-1}}\sup_{r \in (n, 2 n)}   \bigg(1-\frac{n}{r} \bigg)^{-\frac{r}{n}}\|\xi\|^{-n}_{\Pi^{*, \frac{n}{r}}_{r, n} f}  \lesssim  \frac{1}{|\Omega|} \bigg(\frac{\text{diam } \Omega}{\min\{1, A\} A} \bigg)^n.
\end{align*}
Putting this into \eqref{LimP8}, there exists a positive constant $\lambda = \lambda(\|f\|_{L^n}, n, \Omega)$, which depends on $\|f\|_{L^n}, n$ and $\Omega$, such that
$$
	\sup_{r \in  (n, 2 n)} \left[ \exp \bigg(\frac{|f(x)|}{\beta \omega_n^{1/r} \mathcal{G}_r(f)} \bigg)^{n'} - 1 \right]  \lesssim \exp (\lambda |f(x)|)^{n'} - 1,
$$
from which we derive that
$
	\sup_{r \in ( n, 2n )} \Big[  \exp \bigg(\frac{|f(x)|}{\beta \omega_n^{1/r} \mathcal{G}_r(f)} \bigg)^{n'} - 1 \Big] \in L^1
$ (recall that $f \in C(\overline{\Omega})$).
Therefore an application of  the dominated convergence theorem together with the limit formula \eqref{LimP1} yield
\begin{equation*}
	\lim_{r \to n+}		\int_{\R^n} \left[\exp \bigg(\frac{|f(x)|}{\beta \omega_n^{1/r} \mathcal{G}_r(f)} \bigg)^{n'} - 1 \right] \, dx   = \int_{\R^n}  	\left[ \exp \bigg(\frac{|f(x)|}{\beta \omega_n^{1/n} n^{-1/n}  |\Pi^{*}_{n} f|^{-1/n}} \bigg)^{n'} -1 \right] \, dx.
\end{equation*}
\end{proof}

\section{Proof of Theorem \ref{ThmIntro1.12}}\label{Section6}

	The proof of \eqref{SharpFMT} is a combination of Theorem \ref{ThmFractCLYYIntro} with the following inequality (cf. Proposition \ref{PropHolder})
	$$
		n \omega_n^{1+ \frac{n}{r}} \big|\Pi^{*, \frac{n}{r}}_{r, n} f \big|^{-\frac{n}{r}} \leq c_n \|f\|_{B^{\frac{n}{r}}_{r, n}}^n.
	$$
	
	To show \eqref{SharpFMT2}, it is enough to restrict our attention to $r \in (n, 2 n)$. Then the desired assertion immediately follows from Theorem \ref{ThmSharpIlin}.
	
	 Let $e \in \S^{n-1}$ and (cf. \eqref{DefAlpha})
	\begin{equation}\label{CpDef}
		\gamma_{n, r} = \left(\frac{\alpha_{n, r}}{(n+r) \omega_n} \right)^{1/r}.
	\end{equation}
	The proof of \eqref{SharpFMT1} relies on the following  limit formula:
	\begin{equation}\label{ProofSFMT6}
	\lim_{r \to n+} \bigg(1-\frac{n}{r} \bigg)^{\frac{1}{n}} \|f\|_{B^{\frac{n}{r}}_{r, n}} = \frac{\gamma_n}{n^{\frac{1}{n}}} \, \|\nabla f\|_{L^n},
	\end{equation}
	with $\gamma_n = \gamma_{n, n}$.
	The latter follows as an application of Lemma \ref{Th2} with the special choice of $g$ given by
	\begin{equation*}
	g(t, \varepsilon) =
		      \omega(f, t)_{r(\varepsilon)}  = \bigg(\frac{1}{t^n \omega_n} \int_{|h| < t} \|\Delta_h f\|^{r(\varepsilon)}_{L^{r(\varepsilon)}} \, d h \bigg)^{\frac{1}{r(\varepsilon)}},
		      \end{equation*}
	where $\lim_{\varepsilon \to 0+} r(\varepsilon) = n$. More precisely, since $\|\Delta_h f\|_{L^{r(\varepsilon)}} \leq |h| \, \|\nabla f\|_{L^{r(\varepsilon)}}$ (apply the Cauchy--Schwarz inequality to \eqref{314b2}), we have
	\begin{equation*}
		 \frac{\omega(f, t)_{r(\varepsilon)}}{t} \leq \bigg(\frac{c_n}{r(\varepsilon) + n} \bigg)^{\frac{1}{r(\varepsilon)}} \|\nabla f\|_{L^{r(\varepsilon)}},
	\end{equation*}
	which yields
	\begin{equation}\label{ProofSFMT2}
		 \limsup_{\varepsilon \to 0+} \sup_{t \in (0, \infty)} \frac{\omega(f, t)_{r(\varepsilon)}}{t} \leq \bigg(\frac{c_n}{2 n} \bigg)^{\frac{1}{n}}  \|\nabla f\|_{L^n} < \infty.
	\end{equation}

	Furthermore, we claim that
	\begin{equation}\label{6.10}
		\lim_{t \to 0+}  \bigg(\frac{1}{t^{n+r} \omega_n} \int_{|h| < t} \|\Delta_h f\|^{r}_{L^{r}} \, d h \bigg)^{\frac{1}{r}} = \gamma_{n, r} \,  \|\nabla f\|_{L^r}
	\end{equation}
	\emph{uniformly} with respect to $r \in [1, \infty)$. Such a claim is somehow implicit in the computations carried out in \cite[(6.9)]{DLTYY},  however, to make the presentation self-contained, we prefer to give below a direct proof of \eqref{6.10}: Without loss of generality, we may assume that $\text{supp } f \subset B(0, R)$ for some $R > 0$ and $t \in (0, R)$. Let $L$ be the Lipschitz constant of $\nabla f$. Then, for every $x \in \R^n$ and $|h| < t$,
	\begin{equation}\label{6.12}
		|\Delta_h f(x)-\nabla f(x) \cdot h| \leq |h |  \int_0^1 |\nabla f (x + s h) -\nabla f(x)| \, d s \leq \frac{L}{2} |h|^2  \mathbf{1}_{B(0, 2 R)}(x).
	\end{equation}
	In addition, using the definition of $\gamma_{n, r}$ (cf. \eqref{CpDef}), we find that, for any $t > 0$,
	$$
		\gamma_{n, r}  \|\nabla f\|_{L^r} = \bigg(\frac{1}{t^{n + r} \omega_n} \int_{|h| < t} \|\nabla f \cdot h\|^p_{L^r} \, dh \bigg)^{1/r}.
	$$
	From this and \eqref{6.12}, we derive
	\begin{align*}
		\left| \bigg(\frac{1}{t^{n+r} \omega_n} \int_{|h| < t} \|\Delta_h f\|^{r}_{L^{r}} \, d h \bigg)^{1/r} - \gamma_{n, r} \|\nabla f\|_{L^r} \right| & \leq \left( \frac{1}{t^{n+r} \omega_n} \int_{|h| < t} \|\Delta_h f - \nabla f \cdot h\|_{L^r}^r \, dh \right)^{1/r} \\
		&\hspace{-6cm} \leq \frac{L |B(0, 2R)|^{1/r} }{2} \left( \frac{1}{t^{n+r} \omega_n} \int_{|h| < t} |h|^{2 r} \, dh \right)^{1/r} =  \frac{L}{2} \bigg( \frac{n |B(0, 2R)|}{2 r + n} \bigg)^{1/r} t.
	\end{align*}
	Then, taking into account that $( \frac{n |B(0, 2R)|}{2 r + n})^{1/r} \approx 1$ uniformly with respect to $r \in [1, \infty)$, we arrive at the desired assertion \eqref{6.10}.
	
	It follows from \eqref{6.10} (applied to $r = r(\varepsilon)$) that
	\begin{equation*}
		\Big|\frac{g(t, \varepsilon)}{t} - \gamma_n \|\nabla f\|_{L^n} \Big| \leq \Big|\frac{g(t, \varepsilon)}{t} - \gamma_{n, r(\varepsilon)} \|\nabla f\|_{L^{r(\varepsilon)}} \Big| + |\gamma_{n, r(\varepsilon)} \|\nabla f\|_{L^{r(\varepsilon)}}  - \gamma_n \|\nabla f\|_{L^n}| \to 0
	\end{equation*}
	as $t, \varepsilon \to 0+$, i.e.,
	\begin{equation}\label{6.13}
		\lim_{t, \varepsilon \to 0+} \frac{g(t, \varepsilon)}{t} = \gamma_n \|\nabla f\|_{L^n}.
	\end{equation}

	The validity of \eqref{ProofSFMT2} and \eqref{6.13} enables us to apply Lemma \ref{Th2}, namely,
	$$
		\lim_{\varepsilon \to 0} \int_0^\infty \bigg(\frac{g(t, \varepsilon)}{t} \bigg)^n \rho_\varepsilon(t) \, dt = \gamma_n^n \|\nabla f\|_{L^n}^n,
	$$
	where $\rho_\varepsilon (t) = n (1-\frac{n}{r(\varepsilon)}) t^{n(1-\frac{n}{r(\varepsilon)})-1} \mathbf{1}_{(0, 1)}(t)$. This can be expressed as
	$$
		\lim_{\varepsilon \to 0} \bigg(1-\frac{n}{r(\varepsilon)} \bigg) \int_0^1 [t^{-\frac{n}{r(\varepsilon)}} \omega(f, t)_{r(\varepsilon)}]^n \, \frac{dt}{t}= \frac{\gamma_n^n}{n} \,  \|\nabla f\|_{L^n}^n.
	$$
	In particular, since $\omega(f, t)_{r(\varepsilon)} \lesssim \|f\|_{L^{r(\varepsilon)}}$,
	$$
		\lim_{\varepsilon \to 0} \bigg(1-\frac{n}{r(\varepsilon)} \bigg) \int_0^\infty [t^{-\frac{n}{r(\varepsilon)}} \omega(f, t)_{r(\varepsilon)}]^n \, \frac{dt}{t}= \frac{\gamma_n^n}{n} \,  \|\nabla f\|_{L^n}^n,
	$$
	i.e., \eqref{ProofSFMT6} holds.

	According to \eqref{MonMod1} (where the corresponding equivalence constant can be explicitly estimated using \eqref{EquivMod}), we have $$\omega(f, \lambda t)_r \leq 4 (2^n + 1)^{1/r} (1+\lambda) \omega(f, t)_r \qquad \text{for} \qquad \lambda, t > 0.$$ In particular, taking $\lambda = t^{-1}$ in the previous estimate, we have, for every $r > n$,
	\begin{align*}
		 \|f\|_{B^{\frac{n}{r}}_{r, n}} & \geq \left( \int_0^1 \bigg[t^{1-\frac{n}{r}} \,  \frac{\omega(f, t)_r}{t} \bigg]^n \, \frac{dt}{t} \right)^{1/n}  \geq \frac{\omega(f, 1)_p}{8 (2^n + 1)^{1/n}} \, \bigg(\int_0^1 t^{(1-\frac{n}{r}) n} \, \frac{dt}{t} \bigg)^{1/n} \\
		 &  = \Big(1-\frac{n}{r} \Big)^{-1/n} \frac{\omega(f, 1)_r}{8 [n(2^n + 1)]^{1/n}} \geq \Big(1-\frac{n}{r} \Big)^{-1/n} |B(0, R+1)|^{-\frac{1}{n}} \frac{\omega(f, 1)_n}{8 [n(2^n + 1)]^{1/n}},
	\end{align*}
	where we have used H\"older's inequality in the last step.
	As a byproduct, for every $x \in \R^n$,
	\begin{align*}
		\sup_{r > n} \exp \bigg(c_n \,  \frac{|f(x)|}{(1-\frac{n}{r})^{1/n} \|f\|_{B^{n/r}_{r, n}}} \bigg)^{n'} & = \exp \bigg( c_n \,  \frac{|f(x)|}{ \inf_{r > n} (1-\frac{n}{r})^{1/n} \|f\|_{B^{n/r}_{r, n}}} \bigg)^{n'} \\
		& \leq \exp \bigg(\tilde{c}_n \, |B(0, R+1)|^{1/n}  \frac{|f(x)|}{\omega(f, 1)_n} \bigg)^{n'},
	\end{align*}
	with
	$$
		\int_{\R^n} \bigg[ \exp \bigg(\tilde{c}_n \, |B(0, R+1)|^{1/n}   \frac{|f(x)|}{\omega(f, 1)_n} \bigg)^{n'} - 1 \bigg] \, dx < \infty.
	$$
	Therefore, the dominated convergence theorem can be applied, together with \eqref{ProofSFMT6}, in order to get
	\begin{align*}
		\lim_{r \to n+} \int_{\R^n} \bigg[ \exp \bigg(c_n \,  \frac{|f(x)|}{(1-\frac{n}{r})^{1/n} \|f\|_{B^{n/r}_{r, n}}} \bigg)^{n'} - 1 \bigg] \, dx	  = \int_{\R^n} \bigg[\exp \bigg(\frac{n^{1/n} c_n}{\gamma_n} \,  \frac{|f(x)|}{\|\nabla f\|_{L^n}} \bigg)^{n'} - 1 \bigg] \, d x.
	\end{align*}
	The proof of \eqref{SharpFMT1} is finished. \qed

\section{Proof of Theorem \ref{ThmIntro116} and \eqref{Intro121}}\label{Section7}

\subsection{Proof of \eqref{Intro121}} We rely on the following sharp estimate involving classical Besov spaces that has been recently obtained in  \cite[Theorem 6.13, (6.52)]{DominguezTikhonov}. Assume that $\frac{n}{p} < s_0 < s < 1$, then
%	\begin{equation}\label{PAM1}
%		\|f\|_{B^{1-\frac{n}{p} -\lambda}_{\infty, p}} \lesssim \lambda^{1/p} \|f\|_{B^{1-\lambda}_{p, p}} \qquad \text{as} \qquad \lambda \to 0+.
%	\end{equation}
%	 Note that
%	\begin{equation}\label{PAM2}
%		\|f\|_{\dot{B}^\xi_{p, p}} \approx \|f\|_{W^{\xi, p}}
%	\end{equation}
%	with equivalence constants independent of $f$ and $\xi$. Then a simple change of variables ($1-\lambda =s$), together with \eqref{PAM2}, allow to rewrite \eqref{PAM1} as
	\begin{equation}\label{PAM3}
		\|f\|_{B^{s-\frac{n}{p}}_{\infty, p}} \leq C (1-s)^{1/p} \|f\|_{W^{s, p}},
	\end{equation}
	where $C$ is independent of $f$ and $s$.

	Let $u > 0$. We have
	\begin{align*}
		\|f\|_{B^{s-\frac{n}{p}}_{\infty, p}} &\geq  \bigg(\int_u^\infty [t^{-s + \frac{n}{p}} \omega(f, t)_\infty]^p \, \frac{dt}{t} \bigg)^{1/p}  \geq \omega(f, u)_\infty \bigg(\int_u^\infty t^{-s p + n} \frac{dt}{t} \bigg)^{1/p}  = (s p -n)^{-\frac{1}{p}}  \omega(f, u)_\infty  u^{-s + \frac{n}{p}}.
	\end{align*}
	In particular, taking $u = |\text{supp } f|^{1/n}$, we get
	\begin{equation}\label{PAM5}
		 \omega(f,  |\text{supp } f|^{1/n})_\infty \leq  (s p -n)^{\frac{1}{p}} |\text{supp } f|^{\frac{s}{n}-\frac{1}{p}} \|f\|_{B^{s-\frac{n}{p}}_{\infty, p}}.
	\end{equation}
	
	Next we check that
	\begin{equation}\label{new90}
		\|f\|_{L^\infty} \leq c_n  \omega(f,  |\text{supp } f|^{1/n})_\infty.
	\end{equation}
	Given any $x \in \text{supp }f$, we choose $r \in (0, \infty)$ such that $|B(x, r)| = |\text{supp } f|$ (i.e., $\omega_n r^n = |\text{supp }f|$). Then either  $B(x, r) \backslash \text{supp } f \neq \emptyset$ or $\text{supp }f = B(x, r) \cup N$ with $|N| =0$. Assume first that there exists $y \in B(x, r) \backslash \text{supp }f$, then
	$$
		|f(x)| = |f(x)-f(y)| \leq \omega(f, r)_\infty \leq \Big(1+ \frac{1}{\omega_n^{1/n}} \Big) \, \omega(f, |\text{supp }f|^{1/n})_\infty.
	$$
	On the other hand, if $\text{supp }f = B(x, r) \cup N$ with $|N| =0$ then we can take $y \in B(x, 2 r) \backslash \text{supp }f$ and a similar argument as above applies. This proves \eqref{new90}.

	Hence, by \eqref{PAM5} and \eqref{new90},
	\begin{equation*}
		\|f\|_{L^\infty} \leq c_n \, (p -n)^{\frac{1}{p}}  |\text{supp } f|^{\frac{s}{n}-\frac{1}{p}} \|f\|_{B^{s-\frac{n}{p}}_{\infty, p}},
	\end{equation*}
	and combining this with \eqref{PAM3}, we achieve
		\begin{equation*}
			\|f\|_{L^\infty} \lesssim (1-s)^{\frac{1}{p}}  |\text{supp } f|^{\frac{s}{n}-\frac{1}{p}}  \|f\|_{W^{s, p}}.
	\end{equation*}
	This completes the proof of \eqref{Intro121}.
	
	\subsection{Proof of Theorem \ref{ThmIntro116}} It follows from \eqref{Intro121} applied to $f^\star$ (recall that $|\text{supp } f^\star| = |\text{supp } f|$) together with Proposition \ref{Prop348} and Theorem \ref{ThmPSBod} (for the special choice $p=q$) that
	$$
		\|f\|_{L^\infty} \lesssim  (1-s)^{\frac{1}{p}}  |\text{supp } f|^{\frac{s}{n}-\frac{1}{p}} |\Pi^{*, s}_p f^\star|^{-\frac{s}{n}} \leq (1-s)^{\frac{1}{p}}  |\text{supp } f|^{\frac{s}{n}-\frac{1}{p}} |\Pi^{*, s}_p f|^{-\frac{s}{n}}.
	$$
	This proves the desired estimate \eqref{ThmPam1}.
	
	Let $\xi \in \S^{n-1}$. From \eqref{314b2}, we derive
	\begin{align*}
		\|\xi\|^{p s}_{\Pi^{*, s}_p f} &=  \int_0^\infty  t^{-p s  -1}  \|\Delta_{t \xi} f\|_{L^p}^p \, dt  \leq  \int_0^\infty t^{-p s-1}  \min\{2 \|f\|_{L^p}, t \|\nabla f \cdot \xi\|_{L^p}\}^p \, dt \\
		& = \bigg( \int_0^{\frac{2 \|f\|_{L^p}}{\|\nabla f \cdot \xi\|_{L^p}}} t^{p(1-s)-1} \, dt \bigg) \|\nabla f \cdot \xi\|^p_{L^p} + 2^p \bigg( \int_{\frac{2 \|f\|_{L^p}}{\|\nabla f \cdot \xi\|_{L^p}}} ^\infty t^{-p s-1} \, dt \bigg) \|f\|^p_{L^p} \\
		& = \frac{2^{p(1-s)}}{p s(1-s)}\,  \|f\|_{L^p}^{p(1-s)} \|\nabla f \cdot \xi\|_{L^p}^{ p s}  = \frac{2^{p(1-s)}}{p s(1-s)}\,  \|f\|_{L^p}^{p(1-s)}  \|\xi\|_{\Pi^*_p f}^{p s}.
	\end{align*}
	Therefore
	\begin{align*}
		|\Pi^{*, s}_p f| &= \frac{1}{n} \, \int_{\S^{n-1}} \|\xi\|_{\Pi^{*, s}_p f}^{-n} \, d \xi  \gtrsim  (s(1-s))^{\frac{n}{s p}} \, \|f\|_{L^p}^{-\frac{(1-s) n}{s}} \, \frac{1}{n}\int_{\S^{n-1}} \|\xi\|_{\Pi^*_p f}^{-n} \, d \xi  = (s(1-s))^{\frac{n}{s p}} \, \|f\|_{L^p}^{-\frac{(1-s) n}{s}} |\Pi^*_p f|.
	\end{align*}
	This yields (recall that $f \in W^{s, p} \subset L^\infty$)
	\begin{align*}
		|\Pi^{*, s}_p f|^{-\frac{s}{n}}   &\lesssim (s (1-s))^{-\frac{1}{p}} \|f\|_{L^p}^{1-s}  |\Pi^*_p f|^{-\frac{s}{n}} \leq ( s (1-s))^{-\frac{1}{p}} |\text{supp } f|^{\frac{1-s}{p}} \|f\|_{L^\infty}^{1-s}  |\Pi^*_p f|^{-\frac{s}{n}}.
	\end{align*}
	Inserting \eqref{Intro9} in the last estimate, we find
	$$
		|\Pi^{*, s}_p f|^{-\frac{s}{n}}  \lesssim (s (1-s))^{-\frac{1}{p}} |\text{supp } f|^{\frac{1-s}{n}}    |\Pi^*_p f|^{-\frac{1}{n}},
	$$
	i.e., \eqref{ThmPam2} holds.
	
	The relation  \eqref{ThmPam3} is an immediate application of \eqref{HLLimit}. \qed

\end{document}